\documentclass[twoside,12pt]{article}
\usepackage{amsmath,amssymb,amsthm,amsfonts}
\usepackage{paralist}
\usepackage{graphics} 
\usepackage{epsfig} 
\usepackage{graphicx}
\usepackage{multirow}
\usepackage{diagbox}
\usepackage{caption}
\usepackage{subcaption}
\usepackage{epstopdf}
\usepackage[colorlinks=true]{hyperref}
\usepackage{mathrsfs}
\usepackage{subcaption}
\usepackage{txfonts}
\usepackage{enumerate}
\usepackage[numbers,sort&compress]{natbib}
\numberwithin{equation}{section}

\newcommand{\be}{\begin{equation}}
\newcommand{\ee}{\end{equation}}

\newcommand{\ben}{\begin{eqnarray*}}
\newcommand{\enn}{\end{eqnarray*}}

\newtheorem{theorem}{\textbf Theorem}[section]

 \numberwithin{equation}{section}

\newcommand{\pa}{\partial}
\newcommand{\f}{\frac}
\newcommand{\ci}{\cite}

\newcommand{\bd}{\begin{displaymath}}
\newcommand{\ed}{\end{displaymath}}

\begin{document}
\title{{\textbf{An energy-based discontinuous Galerkin method for semilinear wave equations} }}
\author{Daniel Appel\"o \thanks{daniel.appelo@colorado.edu, Department of Applied Mathematics, University of Colorado, Boulder, Boulder CO 80309, United States.  Supported by NSF Grant DMS-1913076.},
Thomas Hagstrom \thanks{thagstrom@smu.edu, Department of Mathematics, Southern Methodist University, Dallas TX 75206, United States. Supported by NSF Grant DMS-1418871.},
Qi Wang \thanks{qwang@swufe.edu.cn.  Department of Mathematics, Southwestern University of Finance and Economics, 555 Liutai Ave, Wenjiang, Chengdu, Sichuan 611130, China.  Supported by Fundamental Research Funds for the Central Universities (JBK2002002 and JBK1805001) and NSF-China (No. 11501460).}
and Lu Zhang \thanks{luzhang@smu.edu,  Department of Mathematics, Southern Methodist University, Dallas TX 75206, United States.  Corresponding author.}
        }
\date{}
\maketitle
\begin{abstract}
We generalize the energy-based discontinuous Galerkin method proposed in \cite{appelo2015new} to
second-order semilinear wave equations. A stability and convergence analysis is presented
along with numerical experiments demonstrating optimal convergence for certain choices of
the interelement fluxes. Applications to the sine-Gordon equation include simulations of
breathers, kink, and anti-kink solitons.
\end{abstract}

\begin{quote}
\end{quote}

\section{Introduction}\label{sec1}

Discontinuous Galerkin methods have emerged as a method-of-choice for
solving hyperbolic initial-boundary value problems in first-order Friedrichs
form. Advantages include guaranteed stability on unstructured grids, local
time evolution, and arbitrary order \ci{HesthavenWarburton08}. However,
typical formulations of wave equations in physics arise as second-order systems.
We believe it is advantageous to directly treat such second-order systems.
First, first-order reformulations require more variables and thus may be less memory efficient
while requiring additional initial and boundary conditions; the latter must be
compatible with the original equations for the first-order reformulation to be equivalent.
The energy-based DG method, introduced in \ci{appelo2015new} and extended to elastic
and advective waves in \ci{el_dg_dath,zhang2019energy}, aims to provide a DG method for
second-order wave equations which is as simple, reliable, flexible, and general as DG
methods applied to Friedrichs systems.

The essential ideas underpinning the energy-based formulation are:
\begin{description}
\item[i.] Introduction of a variable which is \textbf{weakly} equal to the
time derivative of the solution (see (\ref{scheme1}) below),
\item[ii.] Construction of numerical fluxes based on the energy flux at element boundaries.
\end{description}
Advantages of the proposed method are that we use the minimal number of variables
required (compare with LDG \ci{chou2014optimal,AURSAND2017478,CiCP-23-747} and HDG \ci{GodunovHDG} methods) and
simple conservative or upwind fluxes can be chosen to be independent of the mesh
(compare with IPDG \ci{RiviereWheelerWave,GSSwave}). However, although it seems
clear that the method can be adapted to any second-order linear hyperbolic
system, the formulation for nonlinear problems presented in
\ci{appelo2015new} is both incomplete and inconvenient. In particular, the
analogue of (\ref{scheme1}) proposed in \ci{appelo2015new}
involves a nonlinear function of
$\phi_u$. Thus the equation would typically be overdetermined. Moreover, to
guarantee the energy estimate the system must be satisfied for
$\phi_u=u$, which directly leads to a nonlinear problem to calculate
$\f {\pa u}{\pa t}$. Our main result in this work is to show how all these
potential issues can be avoided for semilinear problems.

The remainder of the paper is organized as follows. In Section 2 we introduce the
semidiscretization, proposing a number of interelement fluxes and proving the
basic energy estimate. In Section 3 we prove a suboptimal error estimate
and present several numerical experiments in Section 4. The latter demonstrate
optimal convergence for certain choices of flux: specifically an energy-conserving
alternating flux as well as two energy-dissipating fluxes. We also present
simulations of soliton solutions of the sine-Gordon equation. We summarize our
results in Section 5 and point out areas for future research.

\section{Problem formulation}
We consider semilinear wave equations of the form
\begin{equation}\label{general_problem}
 \frac{\partial^2 u}{\partial t^2} + \theta \frac{\partial u}{\partial t}= c^2\Delta u + f(u),\ \ \ \ {\bf{x}}\in\Omega\subset\mathbb{R}^d, \ \ \ \ t \geq 0,
\end{equation}
where $c > 0$ is the sound wave speed, $\theta\geq0$ is the dissipation coefficient, and $f(u)$ is a smooth function with
$\lim_{u\to 0} \frac{f(u)}{u}$ bounded. The initial conditions are given by
\begin{equation*}
    u({\bf x},0) = g_1({\bf x}),\ \ \ \ \frac{\partial u({\bf x},0)}{\partial t} = g_2({\bf x}), \ \ \ \ {\bf x}\in\Omega\subset\mathbb{R}^d.
\end{equation*}
Note that when $\theta = 0$, (\ref{general_problem}) is the Euler-Lagrange equation derived from the Lagrangian
\begin{equation*}
    L = \frac{1}{2}\left(\frac{\partial u}{\partial t}\right)^2-\frac{c^2}{2}\left|\nabla u\right|^2 - F(u),
\end{equation*}
where $F'(u) = -f(u)$. To derive an energy-based DG formulation for problem (\ref{general_problem}), we introduce a second scalar variable
to produce a system which is first order in time,
\begin{equation}\label{reformulation}
    \left\{\begin{array}{l}
    \frac{\partial u}{\partial t} - v = 0,     \\
\frac{\partial{v}}{\partial t}+\theta v -c^2\Delta u - f(u) = 0.
    \end{array}
    \right.
\end{equation}
The energy takes the form
\begin{equation}\label{energy}
    E = \frac{1}{2}\int_{\Omega} v^2 + c^2\left|\nabla u\right|^2 + 2F(u).
\end{equation}
Note that $F(u) > 0$ corresponds to a defocusing equation and $F(u) < 0$ gives a focusing equation. In the rest of analysis in this paper, we investigate the defocusing equation with $F(u) > 0$, although the method formulation applies in either case. The change of the energy is given by boundary contributions and
a volume integral related to the dissipation:
\begin{equation}\label{energy_t}
    \frac{dE}{dt} = -\theta\int_{\Omega} \left(\frac{\partial u}{\partial t}\right)^2 d{\bf x}+ \int_{\partial \Omega} c^2 v \nabla u\cdot {\bf n}dS,
\end{equation}
where ${\bf n}$ is the outward-pointing unit normal.

We note that in our error analysis we will make the stronger assumption $\f {f(u)}{u} < 0$, which can be enforced after a transformation of
variables if we only assume the ratio is bounded. Then the defocusing assumption holds since
\bd
F(u) = - \int_0^u f(z) dz = \int_0^u \left( - \f {f(z)}{z} \right) zdz > 0.
\ed
\subsection{Semi-discrete DG formulation}\label{sec:semidiscrete}
We develop an energy-based DG scheme for problem (\ref{general_problem}) through the reformulation (\ref{reformulation}). Let the domain $\Omega$ be discretized by non-overlapping elements $\Omega_j$; $\Omega = \cup_j \Omega_j$. Choose the components of the approximations, $(u^h,v^h)$ to $(u,v)$, restricted to $\Omega_j$, to be polynomials
or tensor-product polynomials of degree $q$ and $s$ respectively\footnote{For simplcity we abuse notation and
let $\Pi^r$ denote either the polynomials of degree $r$ or the tensor-product polynomials of degree $r$ in each
coordinate on a reference element.},
\begin{equation*}
    U_h^q = \Big\{u^h({\bf x},t): u^h({\bf x},t)\in \Pi^q(\Omega_j), \ \ {\bf x} \in \Omega_j ,\ \ t\geq0\Big\}, \ \ \ V_h^s = \Big\{v^h({\bf x},t): v^h({\bf x},t)\in \Pi^s(\Omega_j), \ \ {\bf x} \in \Omega_j,\ \ t\geq0\Big\}.
\end{equation*}
We seek an approximation to the system (\ref{reformulation}) which satisfies a discrete energy estimate analogous to (\ref{energy}). Consider a discrete energy in $\Omega_j$,
\begin{equation}\label{discrete_energy1}
    E^h_j(t) = \frac{1}{2}\int_{\Omega_j} \left(v^h\right)^2 + c^2\left|\nabla u^h\right|^2\ d{\bf x} + \sum_{\bf k} \omega_{{\bf k},j}F(u^h({\bf x}_{{\bf k},j})),
\end{equation}
and its time derivative
\begin{equation}\label{discrete_energy_t}
  \frac{dE_j^h(t)}{dt} = \int_{\Omega_j} v^h\frac{\partial v^h}{\partial t} + c^2\nabla u^h \cdot \nabla \frac{\partial u^h}{\partial t}\  d{\bf x} - \sum_{\bf k} \omega_{{\bf k},j}f(u^h({\bf x}_{{\bf k},j}))\frac{\partial u^h}{\partial t}({\bf x}_{{\bf k},j}),
\end{equation}
where we have omitted $t$ in $u^h({\bf x}_{{\bf k},j})$ for simplicity. Note here we use a quadrature rule
with nodes ${\bf x}_{{\bf k},j}$ in $\Omega_j$ and positive weights $\omega_{{\bf k},j}$ to approximate the integration of the nonlinear terms; in our
experiments we use $16$-point Gauss rules.
To obtain a weak form which is compatible with the discrete energy (\ref{discrete_energy1}) and (\ref{discrete_energy_t}), we choose $\phi_u\in U_h^q$, $\phi_v\in V_h^s$ and test the first equation of (\ref{reformulation}) with $-c^2\Delta \phi_u$, the second equation of (\ref{reformulation}) with $\phi_v$ and add terms which vanish for the continuous problem. This yields the following equations,
\begin{equation*}
  \int_{\Omega_j}-c^2\Delta \phi_u \left(\frac{\partial u^h}{\partial t}- v^h\right)\ d {\bf x} = \int_{\partial \Omega_j}c^2\nabla \phi_u\cdot{\bf n} \left(v^{\ast} - \frac{\partial u^h}{\partial t}\right)\ dS + \sum_{\bf k}\omega_{{\bf k},j}\phi_u({\bf x}_{{\bf k},j})\frac{f(u^h({\bf x}_{{\bf k},j}))}{u^h({\bf x}_{{\bf k},j})}\left(\frac{\partial u^h}{\partial t}({\bf x}_{{\bf k},j}) - v^h({\bf x}_{{\bf k},j})\right),
\end{equation*}
\begin{equation*}
\int_{\Omega_j} \phi_v\frac{\partial{v^h}}{\partial t} -c^2\phi_v\Delta u^h + \theta\phi_v v^h\ d{\bf x} - \sum_{\bf k} \omega_{{\bf k},j}\phi_v({\bf x}_{{\bf k},j})f(u^h({\bf x}_{{\bf k},j})) = \int_{\partial \Omega_j}c^2\phi_v\left((\nabla u)^{\ast}\cdot{\bf n}-\nabla u^h\cdot{\bf n}\right)\ dS,
\end{equation*}
where $v^{\ast}$ and $(\nabla u)^{\ast}$ are numerical fluxes on both interelement and physical boundaries. In what follows, we apply integration by parts to obtain an alternative form,
\begin{equation}\label{scheme1}
  \int_{\Omega_j}c^2\nabla \phi_u \cdot \nabla \left(\frac{\partial u^h}{\partial t}- v^h\right)\ d {\bf x} -\sum_{\bf k}\omega_{{\bf k},j}\phi_u({\bf x}_{{\bf k},j})\frac{f(u^h({\bf x}_{{\bf k},j}))}{u^h({\bf x}_{{\bf k},j})}\left(\frac{\partial u^h}{\partial t}({\bf x}_{{\bf k},j}) - v^h({\bf x}_{{\bf k},j})\right) = \int_{\partial \Omega_j}c^2\nabla \phi_u\cdot{\bf n} \left(v^{\ast} - v^h\right)\ dS,
\end{equation}
and
\begin{equation}\label{scheme2}
  \int_{\Omega_j} \phi_v\frac{\partial{v^h}}{\partial t} +c^2\nabla\phi_v\cdot \nabla u^h +\theta \phi_v v^h \ d{\bf x} - \sum_{\bf k} \omega_{{\bf k},j}\phi_v({\bf x}_{{\bf k},j})f(u^h({\bf x}_{{\bf k},j})) = \int_{\partial \Omega_j}c^2\phi_v(\nabla u)^{\ast}\cdot{\bf n} \ dS.
\end{equation}
Now by setting $\phi_u = u^h$ and $\phi_v = v^h$ and recalling (\ref{discrete_energy_t}) we arrive at
\begin{equation*}
    \frac{dE^h}{dt} = \sum_j\frac{dE_j^h}{dt} =-\sum_j
    \int_{\Omega_j}\theta \left(v^h\right)^2\ d{\bf x} + \sum_{j}\int_{\partial\Omega_j} c^2\nabla \phi_u\cdot{\bf n} \left(v^{\ast} - v^h\right) + c^2\phi_v(\nabla u)^{\ast}\cdot{\bf n} \ dS.
\end{equation*}
Note that if $\frac{f(u)}{u} = 0$, then we need to impose an extra equation which determines the mean value of $\frac{\partial u^h}{\partial t}$,
\begin{equation*}
    \int_{\Omega_j} \tilde{\phi}_u\left(\frac{\partial u^h}{\partial t} - v^h\right)d{\bf x} = 0.
\end{equation*}
Here, $\tilde{\phi}_u$ is an arbitrary constant function and this equation does not affect the energy.

The innovation here in comparison with the weak form proposed in \cite{appelo2015new} is the appearance of
$\phi_u \f {f(u^h)}{u^h}$ instead of $f(\phi_u)$ in (\ref{scheme1}). This
exchange obviously yields an invertible linear
system for computing $\f {\pa u^h}{\pa t}$. The energy estimate still holds as the two terms are identical
for the special choice $\phi_u=u^h$.

\subsection{Fluxes}\label{sec:flux}
To complete the formulation of energy-based DG scheme proposed in Section \ref{sec:semidiscrete}, we must specify the numerical fluxes $v^{\ast}, (\nabla u)^{\ast}$ both at interelement and physical boundaries. Denote $'+'$ to be the trace of data from the outside of the element, $'-'$ to be the trace of data from the inside of the element. Here, we introduce the common notation for averages and jumps,
\begin{equation*}
    \{\{v^h\}\} \equiv \frac{1}{2}(v^h)^{+}+\frac{1}{2}(v^h)^{-}, \ \ \ \ [[v^h]] \equiv (v^h)^{+}{\bf n}^+ + (v^h)^{-}{\bf n}^-,
\end{equation*}
and
\begin{equation*}
    \{\{\nabla u^h\}\} \equiv \frac{1}{2}(\nabla u^h)^{+}+\frac{1}{2}(\nabla u^h)^{-}, \ \ \ \ [[\nabla u^h]] \equiv (\nabla u^h)^{+}\cdot{\bf n}^+ + (\nabla u^h)^{-}\cdot{\bf n}^-.
\end{equation*}

\subsubsection{Interelement boundaries}
To analyze the problem, we label two elements sharing one interelement boundary by $1$ and $2$. Then, besides
the volume dissipation,if any, their net contribution to the discrete energy $E^h(t)$ is the boundary integral of
\begin{equation}\label{J}
 J = c^2\nabla u_1^h\cdot{\bf n}_1 \left(v^{\ast} - v_1^h\right) + c^2v_1^h(\nabla u)^{\ast}\cdot{\bf n}_1 + c^2\nabla u_2^h\cdot{\bf n}_2 \left(v^{\ast} - v_2^h\right) + c^2v_2^h(\nabla u)^{\ast}\cdot{\bf n}_2.
\end{equation}
We first introduce the so-called \emph{central flux},
\begin{equation}\label{centralflux}
    v^\ast \equiv \frac{1}{2}\left(v_1^h + v_2^h\right), \ \ \ \ (\nabla u)^\ast \equiv \frac{1}{2}\left(\nabla u_1^h + \nabla u_2^h\right).
\end{equation}
Plug this back into (\ref{J}) and use ${\bf n}_1 = -{\bf n}_2$. Then we have
\begin{equation*}
  J = \frac{1}{2}\left(c^2\nabla u_1^h\cdot{\bf n}_1 \left(v_2^h - v_1^h\right) + c^2v_1^h\left(\nabla u_1^h + \nabla u_2^h\right)\cdot{\bf n}_1 + c^2\nabla u_2^h\cdot{\bf n}_2 \left(v_1^h - v_2^h\right) + c^2v_2^h\left(\nabla u_1^h + \nabla u_2^h\right)\cdot{\bf n}_2\right) = 0.
\end{equation*}
Second, we propose an \emph{alternating flux},
\begin{equation}\label{alternating1}
    v^\ast \equiv v_1^h, \ \ \ \ (\nabla u)^\ast \equiv \nabla u_2^h,
\end{equation}
or
\begin{equation}\label{alternating2}
    v^\ast \equiv v_2^h, \ \ \ \ (\nabla u)^\ast \equiv \nabla u_1^h.
\end{equation}
Using (\ref{alternating1}) as an example, we have
\begin{equation*}
  J = c^2\nabla u_1^h\cdot{\bf n}_1 \left(v_1^h - v_1^h\right) + c^2v_1^h\nabla u^h_2\cdot{\bf n}_1 + c^2\nabla u_2^h\cdot{\bf n}_2 \left(v_1^h - v_2^h\right) + c^2v_2^h\nabla u_2^h\cdot{\bf n}_2 = 0.
\end{equation*}
If $\theta = 0$, then it is clear that both the central flux (\ref{centralflux})
and the alternating flux (\ref{alternating1}) or (\ref{alternating2}) lead to an energy-conserving energy-based DG scheme since $J = 0$. To develop an energy-dissipating scheme for $\theta = 0$, we introduce a \emph{Sommerfeld flux} which yields $J < 0$ in the presence of jumps. Let us denote a flux splitting parameter by $\xi > 0$ which has the same units as the wave speed $c$ and note that,
\begin{equation*}
    v\nabla u\cdot {\bf n} = \frac{1}{4\xi}(v + \xi\nabla u\cdot {\bf n})^2 - \frac{1}{4\xi}(v - \xi\nabla u\cdot {\bf n})^2.
\end{equation*}
Then we enforce
\begin{equation}\label{solve_sommerfeld}
   \left\{\begin{array}{l}
    v^\ast - \xi(\nabla u)^\ast\cdot{\bf n}_1 =  v_1^h - \xi\left(\nabla u_1^h\right)\cdot{\bf n}_1,\\
     \\
    v^\ast - \xi(\nabla u)^\ast\cdot{\bf n}_2 =  v_2^h - \xi\left(\nabla u_2^h\right)\cdot{\bf n}_2.
\end{array}
\right.
\end{equation}
Solving system (\ref{solve_sommerfeld}) yields
\begin{equation}\label{Sommerfeld}
    v^\ast = \{\{v^h\}\} - \frac{\xi}{2}[[\nabla u^h]], \ \ \ \ (\nabla u)^\ast = \{\{\nabla u^h\}\} - \frac{1}{2\xi}[[v^h]].
\end{equation}
Plugging (\ref{Sommerfeld}) into (\ref{J}) we obtain
\begin{equation*}
    J = -\left(\frac{\xi c^2}{2} [[\nabla u^h]]^2 + \frac{c^2}{2\xi}\big|[[v^h]]\big|^2\right) < 0.
\end{equation*}
Thus we have an energy-dissipating scheme even when $\theta=0$ if the \emph{Sommerfeld flux} is used.

\subsubsection{Physical boundaries}
In this section, we focus on the boundary condition,
\begin{equation}\label{simplified_bdry}
    \gamma \frac{\partial u({\bf x},t)}{\partial t} + \eta \nabla u({\bf x},t)\cdot{\bf n} = 0, \ \ \ {\bf x}\in\partial\Omega,
\end{equation}
where $\gamma^2 + \eta^2 = 1$ and $\gamma,\eta\geq 0$. Note that we have a homogeneous Dirichlet boundary condition if $\eta = 0$ and a homogeneous Neumann boundary condition when $\gamma = 0$.
On the one hand, multiplying (\ref{simplified_bdry}) by $\gamma\nabla u\cdot {\bf n}$ gives
\begin{equation}\label{eq1}
    \gamma^2 \frac{\partial u}{\partial t}\nabla u\cdot{\bf n} + \gamma\eta(\nabla u\cdot n)^2 = 0,
\end{equation}
on the other hand, multiplying (\ref{simplified_bdry}) by $\eta \frac{\partial u}{\partial t}$ yields,
\begin{equation}\label{eq2}
    \eta\gamma \big(\frac{\partial u}{\partial t}\big)^2 + \eta^2\frac{\partial u}{\partial t}\nabla u\cdot{\bf n} = 0.
\end{equation}
Combining (\ref{eq1}) and (\ref{eq2}), from (\ref{energy_t}) we have
\begin{equation*}
    \frac{dE}{dt} = -\int_{\Omega} \theta \left(\frac{\partial u}{\partial t}\right)^2d{\bf x}-\int_{\partial\Omega} \gamma\eta\left(\big(\frac{\partial u}{\partial t} \big)^2 + (\nabla u\cdot{\bf n})^2\right)\ dS \leq 0.
\end{equation*}
The numerical fluxes $v^\ast$ and $(\nabla u)^\ast$ are chosen to be consistent with the physical boundary condition (\ref{simplified_bdry}),
\begin{equation*}
    \gamma v^\ast + \eta (\nabla u)^\ast\cdot{\bf n} = 0.
\end{equation*}
By a similar analysis as in \cite{appelo2015new}, we find that if we choose
\begin{equation*}
    v^\ast = v^h - (\gamma - a\eta)\rho, \ \ \ (\nabla u)^\ast = \nabla u^h - (\eta + a\gamma)\rho{\bf n},
\end{equation*}
with
\begin{equation*}
    \rho = \gamma v^h + \eta \left(\nabla u^h\right)\cdot{\bf n} ,
\end{equation*}
then the contribution to the discrete energy from the physical boundaries is given by
\begin{equation*}
\frac{dE^h}{dt}\Bigg|_{\partial \Omega} = -\int_{\partial \Omega}\gamma\eta\Big((v^\ast)^2 + \big((\nabla u)^\ast\cdot{\bf n}\big)^2\Big) + \rho^2\big((1-a^2)\gamma\eta+a(\gamma-\eta)\big),
\end{equation*}
which yields a nonincreasing contribution to the energy if
\begin{equation*}
  b = (1-a^2)\gamma\eta+a(\gamma-\eta) \geq 0.
\end{equation*}

\subsection{Stability of the scheme}
We are now ready to state the stability of the proposed energy-based DG scheme. To make the statement concise,
we introduce a general formulation for the fluxes on the interelement boundaries,
\begin{equation}\label{flux_general}
    v^\ast \equiv \alpha v_1^h + (1-\alpha) v_2^h - \tau [[\nabla u^h]], \ \ \ \ (\nabla u)^\ast \equiv (1-\alpha) \nabla u_1^h + \alpha \nabla u_2^h - \beta [[v^h]],
\end{equation}
with $0\leq\alpha\leq1$ and $\beta, \tau\geq0$. Note that the previous cases correspond to :

\emph{Central flux} : $\alpha = 0.5, \tau = \beta = 0$.

\emph{Alternating flux} : $\alpha = 0, \tau = \beta = 0$ or $\alpha = 1, \tau = \beta = 0$.

\emph{Sommerfeld flux} : $\alpha = 0.5, \tau = \frac{\xi}{2}, \beta = \frac{1}{2\xi}$.\\
For the general flux formulation (\ref{flux_general}), we find that the contribution to the discrete energy from the interelement
boundaries is the boundary integral of
\begin{equation*}
    J = -\left(\ \beta |[[v^h]]|^2 + \tau[[\nabla u^h]]^2\ \right)\leq0.
\end{equation*}

\begin{theorem}
The discrete energy $E^h(t) = \sum_jE^h_j(t)$ with $E_j^h(t)$ defined in (\ref{discrete_energy1}) satisfies
\begin{equation*}
    \frac{dE^h}{dt} = -\sum_j\int_{\Omega_j} \theta \left(v^h\right)^2d{\bf x}-\sum_j\int_{F_j} \left(\ \beta |[[v^h]]|^2 + \tau[[\nabla u^h]]^2\ \right)\ dS -\int_{\partial \Omega}\gamma\eta\Big((v^\ast)^2 + \big((\nabla u)^\ast\cdot{\bf n}\big)^2\Big) + b\rho^2\ dS .
\end{equation*}
If the flux parameters $\tau,\beta$ and $b$ are non-negative, then $E^h(t)\leq E^h(0)$.
\end{theorem}

\section{Error estimates}
To analyze the numerical error of the scheme, we define the errors by
\begin{equation}\label{error_uv}
    e_u = u - u^h, \ \ \ \ e_v = v - v^h,
\end{equation}
and compare $(u^h,v^h)$ with an arbitrary polynomial $(\tilde{u}^h,\tilde{v}^h)$, $\tilde{u}^h\in U_h^q$, $\tilde{v}^h\in V_h^s$ with $q-2\leq s\leq q$. To proceed, we denote the difference
\begin{equation*}
    \tilde{e}_u = \tilde{u}^h - u^h,\ \ \ \ \tilde{e}_v = \tilde{v}^h - v^h,\ \ \ \ \delta_u = \tilde{u}^h - u,\ \ \ \ \delta_v = \tilde{v}^h - v,
\end{equation*}
and the numerical error energy
\begin{eqnarray}
    \mathcal{E} & = & \sum_j\int_{\Omega_j} \frac{1}{2}c^2|\nabla \tilde{e}_u|^2 + \frac{1}{2}\tilde{e}_v^2\ d{\bf x}-\frac{1}{2} \sum_{{\bf k},j}\omega_{{\bf k},j}\frac{f(\tilde{u}^h({\bf x}_{{\bf k},j})}{\tilde{u}^h({\bf x}_{{\bf k},j})}\tilde{e}_u^2({\bf x}_{{\bf k},j}) \label{error_energy} \\ & &
-\sum_{{\bf k},j} \int_0^{\tilde{e}_u({\bf x}_{{\bf k},j})}\omega_{{\bf k},j}\left(\frac{f(\tilde{u}^h({\bf x}_{{\bf k},j})-z)}{\tilde{u}^h({\bf x}_{{\bf k},j})-z}-\frac{f(\tilde{u}^h({\bf x}_{{\bf k},j}))}{\tilde{u}^h({\bf x}_{{\bf k},j})}\right)zdz. \nonumber
\end{eqnarray}
Here we assume $\frac{f(u)}{u} < 0$. However, this restriction can be relaxed as we show in the Remark below.
First we claim that $\mathcal{E}$ is non-negative for small errors. Note that for small $z$
\begin{equation*}
\frac{f(\tilde{u}^h({\bf x}_{{\bf k},j})-z)}{\tilde{u}^h({\bf x}_{{\bf k},j})-z}-\frac{f(\tilde{u}^h({\bf x}_{{\bf k},j}))}{\tilde{u}^h({\bf x}_{{\bf k},j})} = -z\frac{d}{dw}\left(\frac{f(w)}{w}\right)\Bigg|_{w=\tilde{u}^h({\bf x}_{{\bf x},j})-\vartheta z},\ \ \ \ \vartheta \in [0,1] .
\end{equation*}
Then
\begin{equation*}
 \int_0^{\tilde{e}_u({\bf x}_{{\bf k},j})}\omega_{{\bf k},j}\left(\frac{f(\tilde{u}^h({\bf x}_{{\bf k},j})-z)}{\tilde{u}^h({\bf x}_{{\bf k},j})-z}-\frac{f(\tilde{u}^h({\bf x}_{{\bf k},j}))}{\tilde{u}^h({\bf x}_{{\bf k},j})}\right)zdz = \mathcal{O}\left(\tilde{e}^3_u({\bf x}_{{\bf k},j})\right).
\end{equation*}
Thus
\begin{equation*}
  \left|\int_0^{\tilde{e}_u({\bf x}_{{\bf k},j})}\omega_{{\bf k},j}\left(\frac{f(\tilde{u}^h({\bf x}_{{\bf k},j})-z)}{\tilde{u}^h({\bf x}_{{\bf k},j})-z}-\frac{f(\tilde{u}^h({\bf x}_{{\bf k},j}))}{\tilde{u}^h({\bf x}_{{\bf k},j})}\right)zdz\right|<-\frac{1}{2} \sum_{\bf k}\omega_{{\bf k},j}\frac{f(\tilde{u}^h({\bf x}_{{\bf k},j})}{\tilde{u}^h({\bf x}_{{\bf k},j})}\tilde{e}_u^2({\bf x}_{{\bf k},j}),
\end{equation*}
which guarantees the positivity of $\mathcal{E}$. Since both the continuous solution $(u,v)$ and the numerical solution $(u^h,v^h)$ satisfy (\ref{scheme1}) and (\ref{scheme2}), we have
\begin{equation}\label{scheme_e_1}
  \int_{\Omega_j}c^2\nabla \phi_u \cdot \nabla \left(\frac{\partial e_u}{\partial t}- e_v\right)\ d {\bf x} +\sum_{\bf k}\omega_{{\bf k},j}\phi_u({\bf x}_{{\bf k},j})\frac{f(u^h({\bf x}_{{\bf k},j}))}{u^h({\bf x}_{{\bf k},j})}\left(\frac{\partial u^h}{\partial t}({\bf x}_{{\bf k},j}) - v^h({\bf x}_{{\bf k},j})\right) = \int_{\partial \Omega_j}c^2\nabla \phi_u\cdot{\bf n} \left(e_v^{\ast} - e_v\right)\ dS,
\end{equation}
and, using (\ref{error_uv}),
\begin{equation}\label{scheme_e_2}
  \int_{\Omega_j} \phi_v\frac{\partial{e_v}}{\partial t} +c^2\nabla\phi_v\cdot \nabla e_u +\theta \phi_v e_v \ d{\bf x} - \sum_{\bf k} \omega_{{\bf k},j}\phi_v({\bf x}_{{\bf k},j})\left(f(u({\bf x}_{{\bf k},j})-f(u^h({\bf x}_{{\bf k},j}))\right)= \int_{\partial \Omega_j}c^2\phi_v(\nabla e_u)^{\ast}\cdot{\bf n} \ dS .
\end{equation}
Now, by using the relations $e_u = \tilde{e}_u - \delta_u$, $e_v = \tilde{e}_v - \delta_v$, choosing $\phi_u = \tilde{e}_u$, $\phi_v = \tilde{e}_v$ and then summing (\ref{scheme_e_1}) and (\ref{scheme_e_2}), we obtain
\begin{multline}\label{error_es_1}
    \int_{\Omega_j} c^2\nabla\tilde{e}_u\cdot\nabla\frac{\partial\tilde{e}_u}{\partial t} + \tilde{e}_v\frac{\partial\tilde{e}_v}{\partial t}\ d{\bf x} = \int_{\Omega_j} c^2\nabla\tilde{e}_u\cdot\nabla\left(\tilde{e}_v+\left(\frac{\partial\delta_u}{\partial t}-\delta_v\right)\right) + \tilde{e}_v\frac{\partial\delta_v}{\partial t}-c^2\nabla\tilde{e}_v\cdot\nabla(\tilde{e}_u-\delta_u) - \theta\tilde{e}_v(\tilde{e}_v-\delta_v)\ d{\bf x} \\
    -\sum_{\bf k}\omega_{{\bf k},j}\tilde{e}_u({\bf x}_{{\bf k},j})\frac{f(u^h({\bf x}_{{\bf k},j}))}{u^h({\bf x}_{{\bf k},j})}\left(\frac{\partial u^h({\bf x}_{{\bf k},j})}{\partial t}-v^h({\bf x}_{{\bf k},j})\right)-\omega_{{\bf k},j}\tilde{e}_v({\bf x}_{{\bf k},j})\left(f(u({\bf x}_{{\bf k},j}))-f(u^h({\bf x}_{{\bf k},j}))\right)\\
    +\int_{\partial\Omega_j}c^2\nabla\tilde{e}_u\cdot{\bf n}(\tilde{e}_v^\ast-\delta_v^\ast-(\tilde{e}_v-\delta_v))+c^2\tilde{e}_v\left((\nabla\tilde{e}_u)^\ast\cdot{\bf n}-(\nabla\delta_u)^\ast\cdot{\bf n}\right)\ dS.
\end{multline}
An integration by parts in the volume integral $\int_{\Omega_j}c^2\nabla\tilde{e}_u\cdot\nabla\delta_v\ d{\bf x}$ simplifies (\ref{error_es_1}) to
\begin{multline}\label{error_es_2}
  \int_{\Omega_j} c^2\nabla\tilde{e}_u\cdot\nabla\frac{\partial\tilde{e}_u}{\partial t} + \tilde{e}_v\frac{\partial\tilde{e}_v}{\partial t}\ d{\bf x} = \int_{\Omega_j} c^2\nabla\tilde{e}_u\cdot\nabla\frac{\partial\delta_u}{\partial t}+c^2\Delta\tilde{e}_u\delta_v + \tilde{e}_v\frac{\partial\delta_v}{\partial t}+c^2\nabla\tilde{e}_v\cdot\nabla\delta_u- \theta\tilde{e}_v(\tilde{e}_v-\delta_v)\ d{\bf x}  \\
    -\sum_{\bf k}\omega_{{\bf k},j}\tilde{e}_u({\bf x}_{{\bf k},j})\frac{f(u^h({\bf x}_{{\bf k},j}))}{u^h({\bf x}_{{\bf k},j})}\left(\frac{\partial u^h({\bf x}_{{\bf k},j})}{\partial t}-v^h({\bf x}_{{\bf k},j})\right)-\omega_{{\bf k},j}\tilde{e}_v({\bf x}_{{\bf k},j})\left(f(u({\bf x}_{{\bf k},j}))-f(u^h({\bf x}_{{\bf k},j}))\right)\\
    +\int_{\partial\Omega_j}c^2\nabla\tilde{e}_u\cdot{\bf n}(\tilde{e}_v^\ast-\tilde{e}_v)+c^2\tilde{e}_v(\nabla\tilde{e}_u)^\ast\cdot{\bf n}-c^2\nabla\tilde{e}_u\cdot{\bf n}\delta_v^\ast-c^2\tilde{e}_v(\nabla\delta_u)^\ast\cdot{\bf n}\ dS.
\end{multline}
We now must choose $(\tilde{u}^h,\tilde{v}^h)$ to achieve an acceptable error. On $\Omega_j$, we impose for all time $t$ and $\forall\phi_u\in U_h^q, \forall\phi_v\in U_h^s$,
\begin{equation*}
    \int_{\Omega_j}\nabla\phi_u\cdot\nabla\delta_u\ d{\bf x} = 0,\ \ \ \ \int_{\Omega_j} \phi_v\delta_v\ d{\bf x} = 0,\ \ \ \ \int_{\Omega_j}\delta_u\ d{\bf x} = 0,
\end{equation*}
then equation (\ref{error_es_2}) yields
\begin{multline}\label{error_es_3}
  \int_{\Omega_j}  c^2\nabla\tilde{e}_u\cdot\nabla\frac{\partial\tilde{e}_u}{\partial t} + \tilde{e}_v\frac{\partial\tilde{e}_v}{\partial t}\ d{\bf x}
  =\\
-\sum_{\bf k}\omega_{{\bf k},j}\tilde{e}_u({\bf x}_{{\bf k},j})\frac{f(u^h({\bf x}_{{\bf k},j}))}{u^h({\bf x}_{{\bf k},j})}\left(\frac{\partial u^h({\bf x}_{{\bf k},j})}{\partial t}-v^h({\bf x}_{{\bf k},j})\right)
-\omega_{{\bf k},j}\tilde{e}_v({\bf x}_{{\bf k},j})\left(f(u({\bf x}_{{\bf k},j}))-f(u^h({\bf x}_{{\bf k},j}))\right)\\
  -\int_{\Omega_j}\theta\tilde{e}_v^2 \ d{\bf x} +\int_{\partial\Omega_j}c^2\nabla\tilde{e}_u\cdot{\bf n}(\tilde{e}_v^\ast-\tilde{e}_v)
   +c^2\tilde{e}_v(\nabla\tilde{e}_u)^\ast\cdot{\bf n}-c^2\nabla\tilde{e}_u\cdot{\bf n}\delta_v^\ast-c^2\tilde{e}_v(\nabla\delta_u)^\ast\cdot{\bf n}\ dS\\
    = -\int_{\Omega_j}\theta\tilde{e}_v^2 \ d{\bf x}
    +\sum_{\bf k}\omega_{{\bf k},j}\tilde{e}_v({\bf x}_{{\bf k},j})\left(f(u({\bf x}_{{\bf k},j}))-f(u^h({\bf x}_{{\bf k},j}))\right)\\
+\sum_{\bf k}\omega_{{\bf k},j}\tilde{e}_u({\bf x}_{{\bf k},j})\frac{f(u^h({\bf x}_{{\bf k},j}))}{u^h({\bf x}_{{\bf k},j})}\left(\frac{\partial \tilde{e}_u({\bf x}_{{\bf k},j})}{\partial t}-\tilde{e}_v({\bf x}_{{\bf k},j})\right)-\omega_{{\bf k},j}\tilde{e}_u({\bf x}_{{\bf k},j})\frac{f(u^h({\bf x}_{{\bf k},j}))}{u^h({\bf x}_{\bf k}^j)}\left(\frac{\partial \delta_u({\bf x}_{{\bf k},j})}{\partial t}-\delta_v({\bf x}_{{\bf k},j})\right)\\
+\int_{\partial\Omega_j}c^2\nabla\tilde{e}_u\cdot{\bf n}(\tilde{e}_v^\ast-\tilde{e}_v)+c^2\tilde{e}_v(\nabla\tilde{e}_u)^\ast\cdot{\bf n}-c^2\nabla\tilde{e}_u\cdot{\bf n}\delta_v^\ast-c^2\tilde{e}_v(\nabla\delta_u)^\ast\cdot{\bf n}\ dS.
\end{multline}
First, we estimate the third term on the right-hand side of (\ref{error_es_3}),
\begin{multline*}
 \sum_{\bf k}\omega_{{\bf k},j}\tilde{e}_u({\bf x}_{{\bf k},j})\frac{f(u^h({\bf x}_{{\bf k},j}))}{u^h({\bf x}_{{\bf k},j})}\left(\frac{\partial \tilde{e}_u({\bf x}_{{\bf k},j})}{\partial t}-\tilde{e}_v({\bf x}_{{\bf k},j})\right)
 =  \sum_{\bf k}\omega_{{\bf k},j}\left(\frac{f(u^h({\bf x}_{{\bf k},j}))}{u^h({\bf x}_{{\bf k},j})}-\frac{f(\tilde{u}^h({\bf x}_{{\bf k},j}))}{\tilde{u}^h({\bf x}_{{\bf k},j})}\right)\tilde{e}_u({\bf x}_{{\bf k},j})\frac{\partial \tilde{e}_u({\bf x}_{{\bf k},j})}{\partial t} \\
 + \sum_{\bf_k}\omega_{{\bf k},j}\frac{f(\tilde{u}^h({\bf x}_{{\bf k},j}))}{\tilde{u}^h({\bf x}_{{\bf k},j})}\tilde{e}_u({\bf x}_{{\bf k},j})\frac{\partial \tilde{e}_u({\bf x}_{{\bf k},j})}{\partial t}-\omega_{{\bf k},j}\frac{f(u^h({\bf x}_{{\bf k},j}))}{u^h({\bf x}_{{\bf k},j})}\tilde{e}_u({\bf x}_{{\bf k},j})\tilde{e}_v({\bf x}_{{\bf k},j}) \\
 =\frac{d}{dt}\sum_{\bf k}\int_0^{\tilde{e}_u({\bf x}_{{\bf k},j})} \omega_{{\bf k},j}\left(\frac{f(\tilde{u}^h({\bf x}_{{\bf k},j})-z)}{\tilde{u}^h({\bf x}_{{\bf k},j}-z)}-\frac{f(\tilde{u}^h({\bf x}_{{\bf k},j}))}{\tilde{u}^h({\bf x}_{{\bf k},j})}\right)z\ dz
 + \frac{1}{2}\frac{d}{dt}\sum_{\bf k}\omega_{{\bf k},j}\frac{f(\tilde{u}^h({\bf x}_{{\bf k},j}))}{\tilde{u}^h({\bf x}_{{\bf k},j})}\tilde{e}_u^2({\bf x}_{{\bf k},j})\\
  -\sum_{\bf k}\omega_{{\bf k},j}\frac{1}{2}\frac{d}{dt}\left(\frac{f(\tilde{u}^h({\bf x}_{{\bf k},j}))}{\tilde{u}^h({\bf x}_{{\bf k},j})}\right)\tilde{e}_u^2({\bf x}_{{\bf k},j})
 +\omega_{{\bf k},j}\frac{f(u^h({\bf x}_{{\bf k},j}))}{u^h({\bf x}_{{\bf k},j})}\tilde{e}_u({\bf x}_{{\bf k},j})\tilde{e}_v({\bf x}_{{\bf k},j})\\
 -\sum_{\bf k}\int_0^{\tilde{e}_u({\bf x}_{{\bf k},j})} \omega_{{\bf k},j} \frac{d}{dt}\left(\frac{f(\tilde{u}^h({\bf x}_{{\bf k},j})-z)}{\tilde{u}^h({\bf x}_{{\bf k},j}-z)}-\frac{f(\tilde{u}^h({\bf x}_{{\bf k},j}))}{\tilde{u}^h({\bf x}_{{\bf k},j})}\right)z \ dz,
\end{multline*}
then substituting this estimate into (\ref{error_es_3}) and recalling (\ref{error_energy}) we conclude:
\begin{eqnarray}
\frac{d\mathcal{E}}{dt} & = &
  \sum_j\int_{\Omega_j}  c^2\nabla\tilde{e}_u\cdot\nabla\frac{\partial\tilde{e}_u}{\partial t} + \tilde{e}_v\frac{\partial\tilde{e}_v}{\partial t}\ d{\bf x} \nonumber \\ & &
-\sum_j\frac{d}{dt}\sum_{\bf k}\int_0^{\tilde{e}_u({\bf x}_{{\bf k},j})} \omega_{{\bf k},j}\left(\frac{f(\tilde{u}^h({\bf x}_{{\bf k},j})-z)}{\tilde{u}^h({\bf x}_{{\bf k},j}-z)}-\frac{f(\tilde{u}^h({\bf x}_{{\bf k},j}))}{\tilde{u}^h({\bf x}_{{\bf k},j})}\right)z\ dz
 -\sum_j\frac{1}{2}\frac{d}{dt}\sum_{\bf k}\omega_{{\bf k},j}\frac{f(\tilde{u}^h({\bf x}_{{\bf k},j}))}{\tilde{u}^h({\bf x}_{{\bf k},j})}\tilde{e}_u^2({\bf x}_{{\bf k},j}) \nonumber \\
 & = & -\sum_j\int_{\Omega_j}\theta\tilde{e}_v^2 \ d{\bf x} -\sum_{{\bf k},j}\omega_{{\bf k},j}\frac{f(u^h({\bf x}_{{\bf k},j}))}{u^h({\bf x}_{{\bf k},j})}\tilde{e}_u({\bf x}_{{\bf k},j})\tilde{e}_v({\bf x}_{{\bf k},j})
    -\omega_{{\bf k},j}\tilde{e}_v({\bf x}_{{\bf k},j})\left(f(u({\bf x}_{{\bf k},j}))-f(u^h({\bf x}_{{\bf k},j}))\right)  \label{error_es_4} \\ & &
-\sum_{{\bf k},j}\omega_{{\bf k},j}\tilde{e}_u({\bf x}_{{\bf k},j})\frac{f(u^h({\bf x}_{{\bf k},j}))}{u^h({\bf x}_{{\bf k},j})}\left(\frac{\partial \delta_u({\bf x}_{{\bf k},j})}{\partial t}-\delta_v({\bf x}_{{\bf k},j})\right)+\int_0^{\tilde{e}_u({\bf x}_{{\bf k},j})} \omega_{{\bf k},j} \frac{d}{dt}\left(\frac{f(\tilde{u}^h({\bf x}_{{\bf k},j})-z)}{\tilde{u}^h({\bf x}_{{\bf k},j}-z)}-\frac{f(\tilde{u}^h({\bf x}_{{\bf k},j}))}{\tilde{u}^h({\bf x}_{{\bf k},j})}\right)z \ dz \nonumber \\ & &
 -\sum_{{\bf k},j}\omega_{{\bf k},j}\frac{1}{2}\frac{d}{dt}\left(\frac{f(\tilde{u}^h({\bf x}_{{\bf k},j}))}{\tilde{u}^h({\bf x}_{{\bf k},j})}\right)\tilde{e}_u^2({\bf x}_{{\bf k},j}) \nonumber \\ & &
+\sum_j\int_{\partial\Omega_j}c^2\nabla\tilde{e}_u\cdot{\bf n}(\tilde{e}_v^\ast-\tilde{e}_v)+c^2\tilde{e}_v(\nabla\tilde{e}_u)^\ast\cdot{\bf n}-c^2\nabla\tilde{e}_u\cdot{\bf n}\delta_v^\ast-c^2\tilde{e}_v(\nabla\delta_u)^\ast\cdot{\bf n}\ dS. \nonumber
  \end{eqnarray}
  Combining the contributions from neighboring elements in (\ref{error_es_4}) we obtain:
 \begin{eqnarray}
\frac{d\mathcal{E}}{dt}  & = &
  \sum_j\int_{\Omega_j}  c^2\nabla\tilde{e}_u\cdot\nabla\frac{\partial\tilde{e}_u}{\partial t} + \tilde{e}_v\frac{\partial\tilde{e}_v}{\partial t}\ d{\bf x} \nonumber \\ & &
-\sum_j\frac{d}{dt} \sum_{\bf k}\int_0^{\tilde{e}_u({\bf x}_{{\bf k},j})}\omega_{{\bf k},j}\left(\frac{f(\tilde{u}^h({\bf x}_{{\bf k},j})-z)}{\tilde{u}^h({\bf x}_{{\bf k},j}-z)}-\frac{f(\tilde{u}^h({\bf x}_{{\bf k},j}))}{\tilde{u}^h({\bf x}_{{\bf k},j})}\right)z\ dz
 - \sum_j\frac{1}{2}\frac{d}{dt}\sum_{\bf k}\omega_{{\bf k},j}\frac{f(\tilde{u}^h({\bf x}_{{\bf k},j}))}{\tilde{u}^h({\bf x}_{{\bf k},j})}\tilde{e}_u^2({\bf x}_{{\bf k},j}) \nonumber \\
& = & -\sum_j\int_{\Omega_j}\theta\tilde{e}_v^2 \ d{\bf x}-\sum_{{\bf k},j}\omega_{{\bf k},j}\frac{1}{2}\frac{d}{dt}\left(\frac{f(\tilde{u}^h({\bf x}_{{\bf k},j}))}{\tilde{u}^h({\bf x}_{{\bf k},j})}\right)\tilde{e}_u^2({\bf x}_{{\bf k},j})
+\omega_{{\bf k},j}\tilde{e}_u({\bf x}_{{\bf k},j})\frac{f(u^h({\bf x}_{{\bf k},j}))}{u^h({\bf x}_{{\bf k},j})}\left(\frac{\partial \delta_u({\bf x}_{{\bf k},j})}{\partial t}-\delta_v({\bf x}_{{\bf k},j})\right) \nonumber \\ & &
 -\sum_{{\bf k},j}\omega_{{\bf k},j}\frac{f(u^h({\bf x}_{{\bf k},j}))}{u^h({\bf x}_{{\bf k},j})}\tilde{e}_u({\bf x}_{{\bf k},j})\tilde{e}_v({\bf x}_{{\bf k},j})
    -\omega_{{\bf k},j}\tilde{e}_v({\bf x}_{{\bf k},j})\left(f(u({\bf x}_{{\bf k},j}))-f(u^h({\bf x}_{{\bf k},j}))\right) \label{error_es_5} \\ & &
    -\sum_{{\bf k},j} \int_0^{\tilde{e}_u({\bf x}_{{\bf k},j})} \omega_{{\bf k},j} \frac{d}{dt}\left(\frac{f(\tilde{u}^h({\bf x}_{{\bf k},j})-z)}{\tilde{u}^h({\bf x}_{{\bf k},j}-z)}-\frac{f(\tilde{u}^h({\bf x}_{{\bf k},j}))}{\tilde{u}^h({\bf x}_{{\bf k},j})}\right)z \ dz \nonumber \\ & &
    -\sum_j\int_{B_j}\gamma\eta\Big((\tilde{e}_v^\ast)^2 + \big((\nabla \tilde{e}_u)^\ast\cdot{\bf n}\big)^2\Big) + b(\gamma\tilde{e}_v+\eta\nabla\tilde{e}_u\cdot{\bf n})^2+c^2\nabla\tilde{e}_u\cdot{\bf n}\delta_v^\ast+c^2\tilde{e}_v(\nabla\delta_u)^\ast\cdot{\bf n} dS \nonumber \\ & &
     -\sum_j\int_{F_j} \left(\ \beta |[[\tilde{e}_v]]|^2 + \tau[[\nabla \tilde{e}_u]]^2\ \right)- c^2[[\nabla\tilde{e}_u]]\delta_v^\ast+c^2[[\tilde{e}_v]]\cdot(\nabla\delta_u)^\ast\ dS. \nonumber
  \end{eqnarray}
Here, $F_j$ represents interelement boundaries and $B_j$ represents physical boundaries. Besides, we introduce the fluxes $\delta_v^\ast, \nabla\delta_u^\ast$ built from $\delta_v, \nabla\delta_u$ according to the specification in Section \ref{sec:flux}. In what follows, $C$ is a constant independent of the solution and element diameter $h$ for a shape-regular mesh. Denote Sobolev norms by $||\cdot||$ and the associated seminorms by $|\cdot|$. We then have the following error estimate.
\begin{theorem}\label{thm2}
Let $\bar{q} = min(q-1,s)$, $q-1\leq s \leq q$, $\frac{f(u)}{u} \leq -L$, $L >0$ be smooth. Then there exist numbers $C_0, C_1$ depending only on $s,q,\xi,\beta,\tau,b$, the bounds of $\frac{d f(u)}{d u},\frac{f(u)}{u}, \frac{d}{dt}\left(\frac{f(u)}{u}\right)$ and the shape regularity of the mesh, such that for smooth solutions $u,v$, time $T$, and $h$ sufficiently small
\begin{multline}\label{theorem2}
    ||\nabla e_u(\cdot,T)||_{L^2(\Omega)}^2 + ||e_v(\cdot,T)||_{L^2(
    \Omega)}^2\\
\leq C_0e^{C_1 T} \max_{t\leq T}\left[h^{2\zeta}\left(|u(\cdot,t)|_{H^{\bar{q}+2}(\Omega)}+|v(\cdot,t)|_{H^{\bar{q}+1}(\Omega)}\right)
+h^{2(s+1)}\left(\left|\frac{\partial u(\cdot,t)}{\partial t}\right|_{H^{s+1}(\Omega)}^2+|v(\cdot,t)|_{H^{s+1}(\Omega)}^2 + \left|u(\cdot,t)\right|_{H^{s+1}(\Omega)}^2\right)\right] ,
\end{multline}
where
\begin{equation*}
\zeta =
   \left\{\begin{array}{l}
    \bar{q},\ \ \ \ \ \ \ \ \ \ \ \beta,\tau,b\geq0,\\
    \bar{q}+\frac{1}{2},\ \ \ \ \beta,\tau,b>0.
\end{array}
\right.
\end{equation*}
\end{theorem}
\begin{proof}
  From the Bramble-Hilbert lemma (e.g., \cite{ciarlet2002finite}), we have for $\bar{q} = \min(q,s-1)$
  \begin{align}
       \left|\left|\delta_u\right|\right|_{L^2(\Omega)}^2 \leq Ch^{2s+2}\left|u(\cdot,t)\right|_{H^{s+1}(\Omega)}^2,\ \ \ \
 \left|\left|\delta_v\right|\right|_{L^2(\Omega)}^2 \leq Ch^{2s+2}\left|v(\cdot,t)\right|_{H^{s+1}(\Omega)}^2,\ \ \ \
      \left|\left|\frac{\partial\delta_u}{\partial t}\right|\right|_{L^2(\Omega)}^2 \leq Ch^{2s+2}\left|\frac{\partial u(\cdot,t)}{\partial t}\right|_{H^{s+1}(\Omega)}^2.\label{ineq}
  \end{align}
  Now we estimate the nonlinear volume integrals in (\ref{error_es_5}).  By the Cauchy-Schwartz inequality, the Cauchy inequality and (\ref{ineq}) we obtain:
\begin{multline}\label{nlest1}
 -\sum_{{\bf k},j}\omega_{{\bf k},j}\frac{1}{2}\frac{d}{dt}\left(\frac{f(\tilde{u}^h({\bf x}_{{\bf k},j}))}{\tilde{u}^h({\bf x}_{{\bf k},j})}\right)\tilde{e}_u^2({\bf x}_{{\bf k},j})
 +\omega_{{\bf k},j}\frac{f(u^h({\bf x}_{{\bf k},j}))}{u^h({\bf x}_{{\bf k},j})}\tilde{e}_u({\bf x}_{{\bf k},j})\tilde{e}_v({\bf x}_{{\bf k},j})+\omega_{{\bf k},j}\tilde{e}_u({\bf x}_{{\bf k},j})\frac{f(u^h({\bf x}_{{\bf k},j}))}{u^h({\bf x}_{{\bf k},j})}\left(\frac{\partial \delta_u({\bf x}_{{\bf k},j})}{\partial t}-\delta_v({\bf x}_{{\bf k},j})\right) \\
 -\sum_{{\bf k},j}\int_0^{\tilde{e}_u({\bf x}_{{\bf k},j})} \omega_{{\bf k},j} \frac{d}{dt}\left(\frac{f(\tilde{u}^h({\bf x}_{{\bf k},j})-z)}{\tilde{u}^h({\bf x}_{{\bf k},j}-z)}-\frac{f(\tilde{u}^h({\bf x}_{{\bf k},j}))}{\tilde{u}^h({\bf x}_{{\bf k},j})}\right)z \ dz  \\
 \leq C \mathcal{E} + C h^{s+1}\sqrt{\mathcal{E}}\left(\left|\frac{\partial u(\cdot,t)}{\partial t}\right|_{H^{s+1}(\Omega)} +|v(\cdot,t)|_{H^{s+1}(\Omega)} \right),
\end{multline}
and
\begin{multline}\label{nlest2}
   \sum_{{\bf k},j} \omega_{{\bf k},j}\tilde{e}_v({\bf x}_{{\bf k},j})\left(f(u({\bf x}_{{\bf k},j}))-f(u^h({\bf x}_{{\bf k},j}))\right) \leq \sum_{{\bf k},j}C\omega_{{\bf k},j}\left|\tilde{e}_v({\bf x}_{{\bf k},j})\right|\left|u({\bf x}_{{\bf k},j})-u^h({\bf x}_{{\bf k},j})\right| \\
   \leq \sum_{{\bf k},j} C\omega_{{\bf k},j}\left|\tilde{e}_v({\bf x}_{{\bf k},j})\right|\left(\left|\tilde{e}_u({\bf x}_{{\bf k},j})\right|+\left|\delta_u({\bf x}_{{\bf k},j})\right|\right)\leq C \mathcal{E} + Ch^{s+1}\sqrt{\mathcal{E}}\left|u(\cdot,t)\right|_{H^{s+1}(\Omega)} .
\end{multline}
Then, using (\ref{nlest1})-(\ref{nlest2}), (\ref{error_es_5}) is simplified to
\begin{multline*}
    \frac{d\mathcal{E}}{dt} \leq C\mathcal{E} + Ch^{s+1}\sqrt{\mathcal{E}}\left(\left|\frac{\partial u(\cdot,t)}{\partial t}\right|_{H^{s+1}(\Omega)} +|v(\cdot,t)|_{H^{s+1}(\Omega)} + \left|u(\cdot,t)\right|_{H^{s+1}(\Omega)} \right)\\
- \sum_j\int_{B_j} c^2\nabla\tilde{e}_u\cdot{\bf n}\delta_v^\ast+c^2\tilde{e}_v(\nabla\delta_u)^\ast\cdot{\bf n} +\gamma\eta\Big((\tilde{e}_v^\ast)^2 + \big((\nabla \tilde{e}_u)^\ast\cdot{\bf n}\big)^2\Big) +  b(\gamma\tilde{e}_v+\eta\nabla\tilde{e}_u\cdot{\bf n})^2 dS \\
     -\sum_j\int_{F_j} \left(\ \beta |[[\tilde{e}_v]]|^2 + \tau[[\nabla \tilde{e}_u]]^2\ \right)- c^2[[\nabla\tilde{e}_u]]\delta_v^\ast+c^2[[\tilde{e}_v]]\cdot(\nabla\delta_u)^\ast\ dS .
\end{multline*}
Now, we only need to consider the boundary integrals. We use the same analysis as in \cite{appelo2015new}
and complete the estimates for the follwoing cases:

\emph{\textbf{Case I:}} $\beta = 0$ or $\tau = 0$,
\begin{multline}\label{theorem_2_pre_1}
   \frac{d\mathcal{E}}{dt}
     \leq C\mathcal{E} + Ch^{s+1}\sqrt{\mathcal{E}}\left(\left|\frac{\partial u(\cdot,t)}{\partial t}\right|_{H^{s+1}(\Omega)} +|v(\cdot,t)|_{H^{s+1}(\Omega)} + \left|u(\cdot,t)\right|_{H^{s+1}(\Omega)} \right)+Ch^{\bar{q}}\sqrt{\mathcal{E}}\left(|u(\cdot,t)|_{H^{\bar{q}+2}(\Omega)}+|v(\cdot,t)|_{H^{\bar{q}+1}(\Omega)}\right).
\end{multline}
Then, combining a direct integration of (\ref{theorem_2_pre_1}) in time with the assumption $\tilde{e}_u = \tilde{e}_v = 0$ at $t = 0$, we obtain
\begin{equation}\label{theorem2_pre}
    \sqrt{\mathcal{E}}(T)\leq C \left( e^{CT}-1 \right) \max_{t\leq T} \left(h^{\bar{q}}\left(|u(\cdot,t)|_{H^{\bar{q}+2}(\Omega)}+|v(\cdot,t)|_{H^{\bar{q}+1}(\Omega)}\right)+h^{s+1}\left(\left|\frac{\partial u(\cdot,t)}{\partial t}\right|_{H^{s+1}(\Omega)} +|v(\cdot,t)|_{H^{s+1}(\Omega)} + \left|u(\cdot,t)\right|_{H^{s+1}(\Omega)} \right)\right),
\end{equation}
since $\tilde{e}_u = e_u + \delta_u$, $\tilde{e}_v = e_v + \delta_v$, then (\ref{theorem2}) follows from the triangle inequality and  (\ref{theorem2_pre}).

\emph{\textbf{Case II:}} $\beta,\tau, b > 0$,
 \begin{multline}\label{theorem_2_pre_2}
   \frac{d\mathcal{E}}{dt}
     \leq C\mathcal{E} + Ch^{s+1}\sqrt{\mathcal{E}}\left(\left|\frac{\partial u(\cdot,t)}{\partial t}\right|_{H^{s+1}(\Omega)} +|v(\cdot,t)|_{H^{s+1}(\Omega)} + \left|u(\cdot,t)\right|_{H^{s+1}(\Omega)} \right)+Ch^{\bar{q}+1/2}\sqrt{\mathcal{E}}\left(|u(\cdot,t)|_{H^{\bar{q}+2}(\Omega)}+|v(\cdot,t)|_{H^{\bar{q}+1}(\Omega)}\right).
\end{multline}
Then again (\ref{theorem2}) with $\zeta = \bar{q}+\frac{1}{2}$ follows directly from an integration in time of (\ref{theorem_2_pre_2}) combined with the triangle inequality.
\end{proof}
\par\noindent
\emph{\textbf{Remark}}: If $\frac{f(u)}{u} \geq 0$ for some $u$ we may introduce a new variable $u = e^{\alpha t}w$, $\alpha > 0$ and use
the energy-based DG scheme to solve for $w$.  Then so long as $\alpha^2 + \alpha\theta - \frac{f(u)}{u}$ is positive the hypotheses above are satisfied and so the energy and error estimates hold. This applies, for example, to the sine-Gordon equation. However, in our numerical experiments we solve for $u$ rather than $w$.
\par\noindent
\emph{\textbf{Remark}}: For 1-dimensional problems, we can improve the error estimate to $h^{s+1}$ by constructing $(\tilde{u}^h,\tilde{v}^h)$ to make the boundary term in (\ref{error_es_5}) vanish as in \cite{appelo2015new}.
\par\noindent
\emph{\textbf{Remark}}: We note that the error estimate appears to be overly conservative for the problems that we consider in the numerical experiments section. There we do not observe worse than linear growth of the error in time.

\section{Numerical experiments}
In this section we present numerical experiments to evaluate the performance of our scheme. In all cases we use a standard modal formulation and use the $L^2$ norm in space to evaluate the error. We present the numerical experiments in both one and two dimensions. For two-dimensional problems, we consider a simple square domain and use the tensor product of the Legendre polynomials to be the basis functions. All the numerical experiments are marched in time by a fourth-order Runge-Kutta (RK4) scheme and the flux splitting parameter is chosen to be $\xi = c = 1$. In all experiments we choose the time step size sufficiently small to guarantee that the temporal error is dominated by the spatial error.

\subsection{Convergence in $1$D}\label{convergence_rate}
In this section we consider the sine-Gordon equation with a dissipating term, i.e., $\theta = 1$. Particularly, to investigate the order of convergence of our method, we solve the problem
\begin{equation*}
    \frac{\partial^2 u}{\partial t^2} + \frac{\partial u}{\partial t} = c^2\frac{\partial^2 u}{\partial x^2} - \sin(u) + f(x,t), \ \ \ \ x\in(-20,20),\ \ \ \ t\geq0,
\end{equation*}
with a standing breather solution
\begin{equation}\label{breather}
\mbox{solution 1:}\ \ \ \ u(x,t) = 4 \arctan \frac{\sqrt{0.75}\cos(0.5t)}{0.5\cosh(\sqrt{0.75}x)}, \ \ \ \ \ x\in(-20,20),\ \ \ \ \ t\geq0.
\end{equation}
The initial conditions, boundary conditions and the external forcing $f(x,t)$ are chosen so that (\ref{breather}) is the exact solution. We note that our theoretical results establish convergence in the
energy norm, but here we investigate the convergence of the solution itself.

As seen below, in our simulations we find the convergence rate for low degrees $q\leq 3$ is not regular for some cases, so for comparison we also give the results for the manufactured solution
\begin{equation}\label{manufactured}
    \mbox{solution 2:}\ \ \ \   u(x,t) = e^{\sin(x-t)},\ \ \ \ x\in(-20,20),\ \ \ \ \ t\geq0.
\end{equation}
The corresponding initial conditions, boundary conditions and external forcing are determined by the manufactured solution (\ref{manufactured}). For these two examples, we use the same space and time discretization, the only difference is the the solution itself.

The discretization is performed on the computational domain $(-20,20)$ with the element vertices $x_j = -20+ (j-1)h$, $j = 1,2,\cdots,N+1$, $h = \frac{40}{N}$. We evolve the discretized problems until the final time $T = 2$ with the time step $\Delta t = 0.075 h/(2\pi)$. We present the $L^2$ error for $u$. The degrees of the approximation space for $u^h$ are set to be $q = (1,2,3,4,5,6)$.

We test four different fluxes: the central flux denoted by C.-flux, the alternating flux with $\alpha = 0$ denoted by A.-flux, the Sommerfeld flux denoted by S.-flux, and the upwind flux in (\ref{flux_general}) with $\alpha = 0$,  $\tau = \frac{\xi}{2}$, $\beta = \frac{1}{2\xi}$ denoted by A.S.-flux. Note that both the C.-flux and A.-flux are energy-conserving methods; both the S.-flux and A.S.-flux are energy-dissipating methods even when $\theta = 0$. We want to point out that $\alpha = 1$ has a similar performance to $\alpha = 0$ in the cases A.-flux and A.S.-flux; thus we only show the results for $\alpha = 0$ in the rest of the paper. We also consider two different approximation spaces: either $u^h$ and $v^h$ in the same space, i.e., $s = q$, or the degree of the approximation space of $v^h$ one less than $u^h$, i.e., $s = q-1$.

In experiments not shown, we observed that the convergence rate was somewhat irregular for all cases when $L^2$ projection was used to compute the initial conditions. One may use a special projection for the initial conditions to solve this problem; see for example the approach in \cite{chou2014optimal} which discusses a projection for the local DG method with the alternating flux. But here, we adopt a simpler idea as in \cite{zhang2019energy}: transform the problem into one with zero initial conditions,
\begin{equation*}
    u(x,t) = u_0(x) + \tilde{u}(x,t),
\end{equation*}
where $u_0(x)=u(x,0)$. Then we get $u$ by numerically solving for $\tilde{u}(x,t)$.

The $L^2$ errors for $u$ and both problems one and two are presented in Tables \ref{convergence_rate_1d_s} through
\ref{convergence_rate_1d_c_con}.
For the energy-dissipating schemes we observe that the convergence rate for $u^h$ is predictable for both problems,
optimal convergence independent of the degree of $v^h$. For the energy-conserving schemes we see that the behavior of the convergence rate
for $u^h$ is predictable when $q \geq 4$ for both problems with the A.-flux, optimal convergence for both $s = q$ and $s = q-1$.
For the central flux we note suboptimal convergence for $u^h$ when $s = q-1, q \geq 3$ and optimal convergence in $u^h$
when $s = q, q \geq 3$. For the lower order schemes ($q \leq 3$) the behavior of the convergence rate for problem 1 is unpredictable for
all fluxes and approximation spaces. Generally speaking, the error levels for both problems with all fluxes are comparable, but the rates of convergence are more predictable for dissipating fluxes with high order approximation ($q \geq 3$).

\begin{table}
\footnotesize
\caption{\scriptsize{$L^2$ errors in $u$ for problem 1 (\ref{breather}) when the S.-flux is used. $q$ is the degree of
$u^h$, $s$ is the degree of $v^h$ and $N$ is the number of the cells with uniform mesh size $h = 40/N$.}}
\begin{center}
\scalebox{0.92}{
  \begin{tabular}{|c|c|c| c c c c c c|}
  \hline
  $q$&$s$&\diagbox{error}{$N$}&400&800&1200&1600&2000&2400\\
  \hline
    \multirow{2}{*}{1} &{0} &\multirow{2}{*}{$||u-u^h||_{L^2}$} &8.80e-03(--) &6.70e-03(0.39) &2.95e-03(2.02) &2.23e-03(0.99) &1.78e-03(0.99) &1.49e-03(0.99)\\
    \cline{2-2}
    ~ &{1} &~ &1.09e-02(--) &5.35e-03(1.03) &3.63e-03(0.96) &2.74e-03(0.97) &2.20e-03(0.98) &1.84e-03(0.98)\\
  \hline
  \multirow{2}{*}{2} &{1} &\multirow{2}{*}{$||u-u^h||_{L^2}$} &4.75e-05(--) &1.03e-03(-4.44) &1.86e-05(9.90) &2.55e-05(-1.09) &1.40e-06(13.03) &2.21e-06(-2.51)\\
    \cline{2-2}
    ~ &{2} &~ &4.85e-05(--) &6.20e-05(-0.35) &1.85e-05(2.98) &2.40e-05(-0.89) &1.39e-06(12.76) &2.06e-05(-2.17)\\
  \hline
   $q$&$s$&\diagbox{error}{$N$}&50&100&200&400&800&1600\\
  \hline
  \multirow{2}{*}{3} &{2} &\multirow{2}{*}{$||u-u^h||_{L^2}$} &1.91e-02(--) &4.23e-05(8.81) &2.49e-06(4.09) &1.70e-07(3.87) &8.77e-09(4.27) &5.41e-10(4.02)\\
    \cline{2-2}
    ~ &{3} &~ &9.81e-03(--) &3.52e-05(8.12) &2.06e-06(4.09) &1.48e-07(3.80) &7.12e-09(4.38) &4.40e-10(4.02)\\
  \hline
  $q$&$s$&\diagbox{error}{$N$}&80&100&120&140&160&180\\
  \hline
    \multirow{2}{*}{4} &{3} &\multirow{2}{*}{$||u-u^h||_{L^2}$} &5.14e-06(--) &1.66e-06(5.07) &6.58e-07(5.07) &3.02e-07(5.06) &1.54e-07(5.06) &8.47e-08(5.05) \\
    \cline{2-2}
    ~ &{4} &~ &4.12e-06(--) &1.34e-06(5.05) &5.32e-07(5.05) &2.44e-07(5.05) &1.25e-07(5.04) &6.89e-08(5.04)\\
  \hline
  \multirow{2}{*}{5} &{4} &\multirow{2}{*}{$||u-u^h||_{L^2}$} &2.76e-07(--) &7.11e-08(6.08) &2.36e-08(6.04) &9.31e-09(6.04) &4.16e-09(6.04) &2.04e-09(6.03)\\
    \cline{2-2}
    ~ &{5} &~  &2.22e-07(--) &5.72e-08(6.07) &1.91e-08(6.02) &7.54e-09(6.02) &3.37e-09(6.02) &1.66e-09(6.02)\\
  \hline
  \multirow{2}{*}{6} &{5} &\multirow{2}{*}{$||u-u^h||_{L^2}$} &1.45e-08(--) &3.14e-09(6.86) &8.71e-10(7.03) &2.95e-10(7.02) &1.16e-10(7.02) &5.06e-11(7.02)\\
    \cline{2-2}
    ~ &{6} &~ &1.18e-08(--) &2.59e-09(6.80) &7.22e-10(7.01) &2.45e-10(7.01) &9.62e-11(7.01) &4.22e-11(7.01)\\
  \hline
\end{tabular}
}
\end{center}
\label{convergence_rate_1d_s}
\end{table}

\begin{table}
\footnotesize
\caption{\scriptsize{$L^2$ errors in $u$ for problem 2 (\ref{manufactured}) when the S.-flux is used. $q$ is the degree of $u^h$, $s$ is the degree of $v^h$ and $N$ is the number of the cells with uniform mesh size $h = 40/N$.}}
\begin{center}
\scalebox{0.92}{
  \begin{tabular}{|c|c|c| c c c c c c|}
  \hline
  $q$&$s$&\diagbox{error}{$N$}&400&800&1200&1600&2000&2400\\
  \hline
    \multirow{2}{*}{1} &{0} &\multirow{2}{*}{$||u-u^h||_{L^2}$} &5.22e-01(--) &2.76e-01(0.92) &1.88e-01(0.95) &1.42e-01(0.98) &1.14e-01(0.98) &9.57e-02(0.98) \\
    \cline{2-2}
    ~ &{1} &~ &6.65e-01(--) &4.18e-01(0.67) &3.23e-01(0.64) &2.71e-01(0.61) &2.38e-01(0.58) &2.14e-01(0.58)\\
  \hline
  \multirow{2}{*}{2} &{1} &\multirow{2}{*}{$||u-u^h||_{L^2}$} &8.25e-04(--) &1.04e-04(2.99) &3.09e-05(2.99) &1.30e-05(3.01) &6.68e-06(2.98) &3.87e-06(3.00)\\
    \cline{2-2}
    ~&{2} &~ &7.08e-04(--) &8.97e-05(2.98) &2.67e-05(2.99) &1.13e-05(2.99) &5.82e-06(2.97) &3.38e-06(2.98)\\
  \hline
  $q$&$s$&\diagbox{error}{$N$}&50&100&200&400&800&1600\\
  \hline
  \multirow{2}{*}{3} &{2} &\multirow{2}{*}{$||u-u^h||_{L^2}$} &1.42e-02(--) &6.62e-04(4.42) &3.33e-05(4.31) &1.92e-06(4.12) &1.17e-07(4.03) &7.26e-09(4.01)\\
    \cline{2-2}
    ~ &{3} &~ &1.09e-02(--) &4.81e-04(4.50) &2.24e-05(4.43) &1.21e-06(4.21) &7.19e-08(4.08) & 4.41e-09(4.03)\\
  \hline
  $q$&$s$&\diagbox{error}{$N$}&80&100&120&140&160&180\\
  \hline
    \multirow{2}{*}{4} &{3} &\multirow{2}{*}{$||u-u^h||_{L^2}$} &6.40e-05(--) &2.07e-05(5.06) &8.26e-06(5.03) &3.81e-06(5.01) &1.95e-06(5.00) &1.08e-06(5.00)\\
    \cline{2-2}
    ~ &{4} &~ &4.35e-05(--) &1.39e-05(5.11) &5.50e-06(5.09) &2.52e-06(5.08) &1.28e-06(5.07) &7.05e-07(5.06)\\
  \hline
  \multirow{2}{*}{5} &{4} &\multirow{2}{*}{$||u-u^h||_{L^2}$} &3.41e-06(--) &8.90e-07(6.02) &2.97e-07(6.01) &1.18e-07(6.01) &5.28e-08(6.01) &2.60e-08(6.00)\\
    \cline{2-2}
    ~ &{5} &~ &2.23e-06(--) &5.77e-07(6.06) &1.92e-07(6.04) &7.56e-08(6.03) &3.38e-08(6.02) &1.67e-08(6.02)\\
    \hline
  \multirow{2}{*}{6} &{5} &\multirow{2}{*}{$||u-u^h||_{L^2}$} &1.95e-07(--) &4.09e-08(6.99) &1.14e-08(7.00) &3.88e-09(7.00) &1.52e-09(7.01) &6.66e-10(7.01)\\
    \cline{2-2}
    ~ &{6} &~ &1.31e-07(--) &2.76e-08(6.98) &7.71e-09(7.00) &2.62e-09(7.01) &1.03e-09(7.02) &4.49e-10(7.02)\\
  \hline
\end{tabular}
}
\end{center}
\label{convergence_rate_1d_s_con}
\end{table}

\begin{table}
\footnotesize
\caption{\scriptsize{$L^2$ errors in $u$ for problem 1 (\ref{breather}) when the A.S.-flux is used. $q$ is the degree of $u^h$, $s$ is the degree of $v^h$ and $N$ is the number of the cells with uniform mesh size $h = 40/N$.}}
\begin{center}
\scalebox{0.92}{
  \begin{tabular}{|c|c|c| c c c c c c|}
  \hline
  $q$&$s$&\diagbox{error}{$N$}&400&800&1200&1600&2000&2400\\
  \hline
    \multirow{2}{*}{1} &{0} &\multirow{2}{*}{$||u-u^h||_{L^2}$} &8.84e-03 (--) &4.32e-03(1.03) &2.92e-03(0.97) &2.20e-03(0.98) &1.77e-03(0.98) &1.48e-03(0.98)\\
    \cline{2-2}
    ~ &{1} &~ &9.14e-02(--) &1.63e-02(2.49) &6.90e-03(2.12) &1.02e-02(-1.37) &4.34e-03(3.85) &3.64e-03(0.97)\\
  \hline
  \multirow{2}{*}{2} &{1} &\multirow{2}{*}{$||u-u^h||_{L^2}$} &4.32e-05(--) &3.64e-05(0.25) &6.45e-06(4.27) &1.07e-05(-1.75) &1.07e-06(10.33) &9.55e-07(0.60)\\
    \cline{2-2}
    ~&{2} &~ &4.69e-05(--) &5.12e-05(-0.13) &1.81e-05(2.56) &2.41e-05(-0.99) &1.39e-06(12.80) &2.03e-06(-2.09)\\
  \hline
  $q$&$s$&\diagbox{error}{$N$}&50&100&200&400&800&1600\\
  \hline
  \multirow{2}{*}{3} &{2} &\multirow{2}{*}{$||u-u^h||_{L^2}$} &7.43e-02(--) &5.89e-05(10.30) &3.94e-06(3.90) &2.54e-07(3.95) &1.48e-08(4.10) &9.18e-10(4.01)\\
    \cline{2-2}
    ~ &{3} &~ &1.34e-02(--) &4.22e-05(8.31) &2.54e-06(4.05) &2.16e-07(3.56) &7.42e-09(4.86) &4.48e-10(4.05)\\
  \hline
  $q$&$s$&\diagbox{error}{$N$}&80&100&120&140&160&180\\
  \hline
    \multirow{2}{*}{4} &{3} &\multirow{2}{*}{$||u-u^h||_{L^2}$} &5.49e-06(--) &1.75e-06(5.13) &6.90e-07(5.10) &3.16e-07(5.07) &1.61e-07(5.05) &8.88e-08(5.04)\\
    \cline{2-2}
    ~ &{4} &~ &4.26e-06(--) &1.36e-06(5.11) &5.38e-07(5.09) &2.46e-07(5.07) &1.25e-07(5.06) &6.91e-08(5.05)\\
  \hline
  \multirow{2}{*}{5} &{4} &\multirow{2}{*}{$||u-u^h||_{L^2}$} &3.70e-07(--) &1.01e-07(5.79) &3.51e-08(5.82) &1.42e-08(5.87) &6.46e-09(5.90) &3.21e-09(5.93)\\
    \cline{2-2}
    ~ &{5} &~ &2.31e-07(--) &5.95e-08(6.07) &1.98e-08(6.04) &7.79e-09(6.04) &3.47e-09(6.04) &1.70e-09(6.04)\\
  \hline
  \multirow{2}{*}{6} &{5} &\multirow{2}{*}{$||u-u^h||_{L^2}$}  &1.77e-08(--) &3.75e-09(6.95) &1.01e-09(7.20) &3.35e-10(7.15) &1.29e-10(7.12) &5.61e-11(7.09)\\
    \cline{2-2}
    ~ &{6} &~ &1.22e-08(--) &2.63e-09(6.86) &7.29e-10(7.04) &2.47e-10(7.03) &9.67e-11(7.02) &4.23e-11(7.02)\\
  \hline
\end{tabular}
}
\end{center}
\label{convergence_rate_1d_as}
\end{table}

\begin{table}
\footnotesize
\caption{\scriptsize{$L^2$ errors in $u$ for problem 2 (\ref{manufactured}) when the A.S.-flux is used. $q$ is the degree of $u^h$, $s$ is the degree of $v^h$ and $N$ is the number of the cells with uniform mesh size $h = 40/N$.}}
\begin{center}
\scalebox{0.92}{
  \begin{tabular}{|c|c|c| c c c c c c|}
  \hline
  $q$& $s$ &\diagbox{error}{$N$}&40&800&1200&1600&2000&2400\\
  \hline
    \multirow{2}{*}{1} &{0} &\multirow{2}{*}{$||u-u^h||_{L^2}$} &5.50e-01(--) &2.91e-01(0.92) &1.97e-01(0.96) &1.48e-01(0.99) &1.18e-01(1.02) &9.88e-02(1.00) \\
    \cline{2-2}
    ~ &{1} &~ &9.92e-01(--) &5.85e-01(0.76) &4.29e-01(0.76) &3.47e-01(0.74) &2.95e-01(0.73)& 2.59e-01(0.71)\\
  \hline
  \multirow{2}{*}{2} &{1} &\multirow{2}{*}{$||u-u^h||_{L^2}$} &6.88e-04(--) &8.64e-05(2.99) &2.56e-05(3.00) &1.08e-05(3.00) &5.55e-06(2.98) &3.21e-06(3.00)\\
    \cline{2-2}
    ~ &{2} &~ &7.07e-04(--) &8.95e-05(2.98) &2.67e-05(2.98) &1.13e-05(2.99) &5.82e-06(2.97) &3.38e-06(2.98)\\
  \hline
  $q$& $s$ &\diagbox{error}{$N$}&50&100&200&400&800&1600\\
  \hline
  \multirow{2}{*}{3} &{2} &\multirow{2}{*}{$||u-u^h||_{L^2}$} &1.37e-02(--) &7.26e-04(4.24) &3.58e-05(4.34) &1.96e-06(4.19) &1.18e-07(4.06) &7.31e-09(4.01)\\
    \cline{2-2}
    ~ &{3} &~ &1.37e-02(--) &7.12e-04(4.26) &3.03e-05(4.55) &1.40e-06(4.44) &7.54e-08(4.21) &4.47e-09(4.07)\\
  \hline
  $q$& $s$ &\diagbox{error}{$N$}&80&100&120&140&160&180\\
  \hline
    \multirow{2}{*}{4} &{3} &\multirow{2}{*}{$||u-u^h||_{L^2}$} &5.66e-05(--) &1.73e-05(5.32) &6.70e-06(5.19) &3.04e-06(5.14) &1.54e-06(5.11) &8.43e-07(5.09)\\
    \cline{2-2}
    ~ &{4} &~ &4.58e-05(--) &1.43e-05(5.23) &5.58e-06(5.15) &2.54e-06(5.11) &1.29e-06(5.09) &7.08e-07(5.08)\\
  \hline
  \multirow{2}{*}{5} &{4} &\multirow{2}{*}{$||u-u^h||_{L^2}$} &2.95e-06(--) &8.09e-07(5.80) &2.80e-07(5.83) &1.13e-07(5.86) &5.16e-08(5.89) &2.57e-08(5.91)\\
    \cline{2-2}
    ~ &{5} &~ &2.36e-06(--) &5.95e-07(6.06) &1.97e-07(6.05) &7.76e-08(6.05) &3.46e-08(6.04) &1.70e-08(6.04)\\
  \hline
   \multirow{2}{*}{6} &{5} &\multirow{2}{*}{$||u-u^h||_{L^2}$} &1.65e-07(--) &3.35e-08(7.13) &9.16e-09(7.11) &3.07e-09(7.09) &1.20e-09(7.07) &5.20e-10(7.06)\\
    \cline{2-2}
    ~ &{6} &~ &1.36e-07(--) &2.82e-08(7.06) &7.81e-09(7.05) &2.64e-09(7.04) &1.03e-09(7.04) &4.50e-10(7.03)\\
  \hline
\end{tabular}
}
\end{center}
\label{convergence_rate_1d_as_con}
\end{table}

\begin{table}
\footnotesize
\caption{\scriptsize{$L^2$ errors in $u$ for problem 1 (\ref{breather}) when the A.-flux is used. $q$ is the degree of
 $u^h$, $s$ is the degree of  $v^h$, and $N$ is the number of the cells with uniform mesh size $h = 40/N$.}}
\begin{center}
\scalebox{0.92}{
  \begin{tabular}{|c|c|c| c c c c c c|}
  \hline
  q&s&\diagbox{error}{N}&400&800&1200&1600&2000&2400\\
  \hline
    \multirow{2}{*}{1}&{0}&\multirow{2}{*}{$||u-u^h||_{L^2}$}&1.78e-03(--)&4.87e-04(1.87)&1.98e-04(2.22)&1.12e-04(1.99)&7.12e-05(2.02)&4.92e-05(2.03)\\
    \cline{2-2}
    ~&{1}&~&6.74e-01(--)&6.73e-01(0.00)&7.03e-01(-0.11)&6.96e-01(0.04)&6.91e-01(0.03)&6.87e-01(0.03)\\
  \hline
  \multirow{2}{*}{2}&{1}&\multirow{2}{*}{$||u-u^h||_{L^2}$}&3.62e-04(--)&2.16e-05(4.06)&6.18e-06(3.09)&8.19e-06(-0.98) &1.10e-06(8.99)&6.80e-07(2.65)\\
    \cline{2-2}
    ~&{2} &~ &3.23e-02(--) &2.58e-03(3.64) &2.86e-04(5.43) &2.98e-03(-8.14) &5.69e-04(7.41) &3.13e-04(3.27)\\
  \hline
  q&s&\diagbox{error}{N}&50&100&200&400&800&1600\\
  \hline
  \multirow{2}{*}{3}&{2}&\multirow{2}{*}{$||u-u^h||_{L^2}$} &3.99e-03(--) &5.67e-05(6.14) &3.60e-06(3.98) &2.21e-07(4.03) &1.37e-08(4.00) &1.34e-09(3.36)\\
    \cline{2-2}
    ~&{3} &~ &1.95e-03(--) &9.72e-05(4.33) &9.94e-06(3.29) &2.02e-06(2.30) &4.31e-08(5.55) &3.18e-09(3.76)\\
  \hline
  q&s&\diagbox{error}{N}&80&100&120&140&160&180\\
  \hline
    \multirow{2}{*}{4}&{3} &\multirow{2}{*}{$||u-u^h||_{L^2}$}  &7.12e-06(--) &2.35e-06(4.96) &9.50e-07(4.98) &4.40e-07(4.99) &2.26e-07(4.99) &1.26e-07(4.99)\\
    \cline{2-2}
    ~&{4} &~&8.64e-06(--) &2.88e-06(4.93) &1.16e-06(4.97) &5.39e-07(4.98) &2.77e-07(4.98) &1.54e-07(4.99)\\
  \hline
  \multirow{2}{*}{5}&{4} &\multirow{2}{*}{$||u-u^h||_{L^2}$} &3.88e-07(--) &1.01e-07(6.02) &3.41e-08(5.97) &1.36e-08(5.98) &6.10e-09(5.99) &3.01e-09(5.99)\\
    \cline{2-2}
    ~&{5} &~ &4.41e-07(--) &1.15e-07(6.02) &3.88e-08(5.97) &1.54e-08(5.98) &6.94e-09(5.98)&3.43e-09(5.99)\\
  \hline
  \multirow{2}{*}{6}&{5} &\multirow{2}{*}{$||u-u^h||_{L^2}$} &2.06e-08(--) &4.57e-09(6.74) &1.28e-09(6.97) &4.37e-10(6.98) &1.72e-10(6.97)&7.63e-11(6.92)\\
    \cline{2-2}
    ~&{6} &~ &2.25e-08(--) &5.02e-09(6.73) &1.41e-09(6.97) &4.80e-10(6.97) &1.90e-10(6.96)&8.41e-11(6.90)\\
  \hline
\end{tabular}
}
\end{center}
\label{convergence_rate_1d_a}
\end{table}

\begin{table}
\footnotesize
\caption{\scriptsize{$L^2$ errors in $u$ for problem 2 (\ref{manufactured}) when the A.-flux is used. $q$ is the degree of
 $u^h$, $s$ is the degree of  $v^h$ and $N$ is the number of the cells with uniform mesh size $h = 40/N$.}}
\begin{center}
\scalebox{0.89}{
  \begin{tabular}{|c|c|c| c c c c c c|}
  \hline
  $q$& $s$ &\diagbox{error }{N}&40&800&1200&1600&2000&2400\\
  \hline
    \multirow{2}{*}{1}&{0} &\multirow{2}{*}{$||u-u^h||_{L^2}$} &5.84e-01(--) &4.25e-01(0.46) &3.50e-01(0.48) &3.05e-01(0.48) &2.73e-01(0.50) &2.50e-01(0.49) \\
    \cline{2-2}
    ~&{1} &~ &4.31e+00(--) &4.34e+00(-0.01) &4.35e+00(-0.01) &4.36e+00(-0.01) &4.36e+00(0.00)& 4.37e+00(-0.00)\\
  \hline
  \multirow{2}{*}{2}&{1} &\multirow{2}{*}{$||u-u^h||_{L^2}$} &7.73e-04(--) &9.79e-05(2.98) &2.92e-05(2.98) &1.24e-05(2.98) &6.36e-06(2.99)&3.69e-06(2.98)\\
    \cline{2-2}
    ~&{2} &~ &1.17e-02(--) &2.92e-03(2.00) &1.30e-03(2.00) &7.30e-04(2.01) &4.67e-04(2.00)&3.24e-04(2.00)\\
  \hline
  $q$& $s$ &\diagbox{error }{N}&50&100&200&400&800&1600\\
  \hline
  \multirow{2}{*}{3}&{2} &\multirow{2}{*}{$||u-u^h||_{L^2}$}&1.46e-02(--) &6.81e-04(4.42) &4.20e-05(4.02) &2.29e-06(4.20) &1.43e-07(4.00)&9.39e-09(3.93)\\
    \cline{2-2}
    ~ &{3} &~ &3.71e-02(--) &2.74e-03(3.76) &1.78e-04(3.94) &1.12e-05(3.99) &7.03e-07(4.00)& 4.39e-08(4.00)\\
  \hline
  $q$& $s$ &\diagbox{error }{N}&80&100&120&140&160&180\\
  \hline
    \multirow{2}{*}{4}&{3} &\multirow{2}{*}{$||u-u^h||_{L^2}$} &8.23e-05(--) &2.59e-05(5.18) &1.06e-05(4.92) &4.96e-06(4.90) &2.41e-06(5.41) &1.32e-06(5.08)\\
    \cline{2-2}
    ~&{4} &~ &9.51e-05(--) &3.05e-05(5.10) &1.21e-05(5.07) &5.55e-06(5.06) &2.83e-06 (5.04)&1.57e-06(5.03)\\
  \hline
  \multirow{2}{*}{5}&{4} &\multirow{2}{*}{$||u-u^h||_{L^2}$} &3.74e-06(--) &1.01e-06(5.88) &3.47e-07(5.84) &1.41e-07(5.86) &6.39e-08(5.91)&3.14e-08(6.03)\\
    \cline{2-2}
    ~ &{5} &~ &4.19e-06(--) &1.10e-06(5.98) &3.69e-07(6.00) &1.47e-07(5.99) &6.57e-08(6.01) &3.24e-08(5.99)\\
  \hline
   \multirow{2}{*}{6}&{5} &\multirow{2}{*}{$||u-u^h||_{L^2}$} &2.18e-07(--) &4.60e-08(6.98) &1.29e-08(6.99) &4.35e-09(7.03) &1.74e-09(6.88) &7.89e-10(6.70)\\
    \cline{2-2}
    ~&{6} &~ &2.31e-07(--) &4.91e-08(6.93) &1.38e-08(6.97) &4.71e-09(6.97) &1.86e-09(6.97)&8.12e-10(7.03)\\
  \hline
\end{tabular}
}
\end{center}
\label{convergence_rate_1d_a_con}
\end{table}

\begin{table}
\footnotesize
\caption{\scriptsize{$L^2$ errors in $u$ for problem 1 (\ref{breather}) when the C.-flux is used. $q$ is the degree of
 $u^h$, $s$ is the degree of  $v^h$ and $N$ is the number of the cells with uniform mesh size $h = 40/N$.}}
\begin{center}
  \scalebox{0.92}{
  \begin{tabular}{|c|c|c| c c c c c c|}
  \hline
  $q$& $s$ &\diagbox{error}{N}&40&800&1200&1600&2000&2400\\
  \hline
    \multirow{2}{*}{1}&{0} &\multirow{2}{*}{$||u-u^h||_{L^2}$}& 2.32e-03(--)&5.80e-04(2.00) &2.58e-04(2.00) &1.45e-04(2.00) &9.28e-05(2.00) &6.44e-05(2.00) \\
    \cline{2-2}
    ~&{1} &~ &3.29e-03(--) &1.17e-03(1.50) &2.98e-04(3.36) &7.82e-04(-3.35) &2.52e-04(5.08) &6.50e-04(-5.21)\\
  \hline
  \multirow{2}{*}{2}&{1} &\multirow{2}{*}{$||u-u^h||_{L^2}$} & 4.99e-03(--)&1.06e-03(2.23) &7.32e-05(6.60) &4.53e-05(1.67) &1.52e-05(4.89) &8.98e-06(2.89)\\
    \cline{2-2}
    ~&{2} &~ & 2.85e-01(--) &4.11e-04(9.44) &3.85e-04(0.16) &3.17e-03(-7.32) &1.73e-03(2.72) &2.74e-04(10.10)\\
  \hline
  $q$& $s$ &\diagbox{error}{N}&50&100&200&400&800&1600\\
  \hline
  \multirow{2}{*}{3}&{2} &\multirow{2}{*}{$||u-u^h||_{L^2}$} & 2.55e-03(--)&3.79e-05(6.07) &1.90e-06(4.32) &1.13e-07(4.07) &6.98e-09(4.02) &4.35e-10(4.00)\\
    \cline{2-2}
    ~&{3} &~ & 2.07e-03(--) &5.44e-05(5.25) &2.20e-06(4.63) &1.19e-07(4.21) &7.04e-09(4.08) &4.36e-10(4.01)\\
  \hline
  $q$& $s$ &\diagbox{error }{N}&80&100&120&140&160&180\\
  \hline
    \multirow{2}{*}{4}&{3} &\multirow{2}{*}{$||u-u^h||_{L^2}$} &9.54e-06(--) &3.95e-06(3.95) &1.91e-06(3.99) &1.03e-06(3.98) &6.06e-07(3.99)& 3.79e-07(3.99)\\
    \cline{2-2}
    ~&{4} &~ &9.93e-06(--) &3.58e-06(4.57) &1.51e-06(4.72) &7.23e-07(4.80) &3.79e-07(4.84)&2.13e-07(4.88)\\
  \hline
    \multirow{2}{*}{5}&{4} &\multirow{2}{*}{$||u-u^h||_{L^2}$} &5.82e-07(--) &9.49e-08(8.13) &2.53e-08(7.25) &9.13e-09(6.62) &3.88e-09(6.41)&1.84e-09(6.31)\\
    \cline{2-2}
    ~&{5} &~ &4.30e-07(--) &9.63e-08(6.70) &2.90e-08(6.58) &1.06e-08(6.54) &4.45e-09(6.49)& 2.09e-09(6.43)\\
  \hline
   \multirow{2}{*}{6}&{5} &\multirow{2}{*}{$||u-u^h||_{L^2}$}&2.13e-08(--) &5.72e-09(5.89) &1.92e-09(5.98) &7.65e-10(5.98) &3.44e-10(5.98)&1.70e-10(5.99)\\
    \cline{2-2}
    ~&{6} &~  &2.17e-08(--) &5.31e-09(6.30) &1.61e-09(6.55) &5.76e-10(6.66) &2.34e-10(6.74)&1.05e-10(6.80)\\
  \hline
\end{tabular}
}
\end{center}
\label{convergence_rate_1d_c}
\end{table}

\begin{table}
\footnotesize
\caption{\scriptsize{$L^2$ errors in $u$ for problem 2 (\ref{manufactured}) when the C.-flux is used. $q$ is the degree of
 $u^h$, $s$ is the degree of  $v^h$ and $N$ is the number of the cells with uniform mesh size $h = 40/N$.}}
\begin{center}
\scalebox{0.92}{
  \begin{tabular}{|c|c|c| c c c c c c|}
  \hline
  $q$& $s$ &\diagbox{error }{N}&400&800&1200&1600&2000&2400\\
  \hline
    \multirow{2}{*}{1}&{0} &\multirow{2}{*}{$||u-u^h||_{L^2}$} &1.01e-01(--) &4.67e-02(1.11) &3.07e-02(1.03) &2.29e-02(1.02) &1.83e-02(1.00)&1.52e-02(1.01) \\
    \cline{2-2}
    ~&{1} &~ &4.44e-01(--) &3.02e-01(0.56) &2.45e-01(0.52) &2.12e-01(0.50) &1.90e-01(0.49) &1.73e-01(0.50)\\
  \hline
  \multirow{2}{*}{2}&{1} &\multirow{2}{*}{$||u-u^h||_{L^2}$} &4.57e-03(--) &1.14e-03(2.00) &5.08e-04(1.99) &2.86e-04(2.00) &1.83e-04(2.00)&1.27e-04(2.00)\\
    \cline{2-2}
    ~&{2} &~ &1.85e-02(--) &4.66e-03(1.99) &2.07e-03(2.00) &1.17e-03(1.98) &7.46e-04(2.02)&5.18e-04(2.00)\\
  \hline
  $q$& $s$ &\diagbox{error }{N}&50&100&200&400&800&1600\\
  \hline
  \multirow{2}{*}{3}&{2} &\multirow{2}{*}{$||u-u^h||_{L^2}$}&4.08e-02(--) &5.12e-04(6.31) &2.38e-05(4.43) &1.42e-06(4.06) &8.84e-08(4.01)&5.51e-09(4.00)\\
    \cline{2-2}
    ~&{3} &~ &3.56e-02(--) &1.06e-03(5.06) &2.78e-05(5.26) &1.22e-06(4.52) &7.13e-08(4.09)& 4.39e-09(4.02)\\
  \hline
  $q$& $s$ &\diagbox{error }{N}&80&100&120&140&160&180\\
  \hline
    \multirow{2}{*}{4}&{3} &\multirow{2}{*}{$||u-u^h||_{L^2}$} &1.19e-04(--) &5.01e-05(3.87) &2.43e-05(3.96) &1.23e-05(4.44) &7.07e-06(4.12)& 4.50e-06(3.84)\\
    \cline{2-2}
    ~&{4} &~ &1.11e-04(--) &3.85e-05(4.75) &1.59e-05(4.84) &7.50e-06(4.89) &3.88e-06(4.93)&2.17e-06(4.95)\\
  \hline
 \multirow{2}{*}{5}&{4} &\multirow{2}{*}{$||u-u^h||_{L^2}$} &7.02e-06(--) &9.44e-07(8.99) &2.69e-07(6.88) &1.00e-07(6.41) &4.27e-08(6.39)&2.02e-08(6.35)\\
    \cline{2-2}
    ~&{5} &~ &4.14e-06(--) &9.47e-07(6.61) &2.84e-07(6.60) &1.04e-07(6.52) &4.39e-08(6.46)& 2.06e-08(6.41)\\
  \hline
   \multirow{2}{*}{6}&{5} &\multirow{2}{*}{$||u-u^h||_{L^2}$}&2.24e-07(--) &5.73e-08(6.11) &1.91e-08(6.02) &7.77e-09(5.85) &3.47e-09(6.04)&1.70e-09(6.02)\\
    \cline{2-2}
    ~&{6} &~ &2.09e-07(--) &5.08e-08(6.34) &1.54e-08(6.54) &5.53e-09(6.66) &2.25e-09(6.73)&1.01e-09(6.79)\\
  \hline
\end{tabular}
}
\end{center}
\label{convergence_rate_1d_c_con}
\end{table}

\subsection{Soliton solutions of the sine-Gordon equation in $1$D}
In this section, we consider the sine-Gordon equation without the dissipating term, i.e., $\theta = 0$,
\begin{equation}\label{sine_gordon_1d}
\frac{\partial^2 u}{\partial t^2} = c^2\frac{\partial^2 u}{\partial x^2} - \sin(u),\ \ \ \ x\in(-20,20), \ \ \ \ t\geq0.
\end{equation}
This equation appears in a number of physical applications and is famous for its soliton and multi-soliton solutions. Here, we focus on investigating these soliton solutions: breather soliton, kink soliton, anti-kink soliton and multi-soliton solutions: kink-kink collision, kink-antikink collision. In the numerical simulations, the number of elements is chosen to be $N = 120$. We impose no-flux conditions at the computational domain boundaries,
\begin{equation*}
   \frac{\partial u}{\partial x}(-20,t) = \frac{\partial u}{\partial x}(20,t) = 0, \ \ \ \ t\geq0.
\end{equation*}

\subsubsection{Standing breather soliton}
To numerically simulate the breather soliton solution of the sine-Gordon equation (\ref{sine_gordon_1d}), we consider the initial conditions,
\begin{equation*}
    u(x,0) = 4\arctan \frac{\sqrt{0.75}}{0.5\cosh(\sqrt{0.75}x)},\ \ \ \ \frac{\partial u}{\partial t}(x,0) = 0,\ \ \ \ x\in(-20,20).
\end{equation*}
These conditions correspond to an exact standing breather soliton solution
\begin{equation*}
    u(x,t) = 4\arctan\frac{\sqrt{0.75}\cos(0.5t)}{0.5\cosh(\sqrt{0.75}x)}.
\end{equation*}
\textbf{\emph{Time history of the numerical energy:}} we first study the numerical energy of the DG approximations to the standing breather solution. As above, we consider the four different fluxes: A.-flux, C.-flux, A.S.-flux and S.-flux; we also consider the cases
where both $u^h, v^h$ are in the same approximation space ($s = q$) and when the degree of the approximation space for $v^h$ is one less than $u^h$ ($s = q-1$). The degree of the approximation space for $u$ is fixed to be $q = 4$. We evolve the numerical solution until $T = 120$ with $h=1/3$ and use the $4$-stage Runge Kutta method with $\Delta t = 0.195h/(2\pi)$.

\begin{figure}[h!]
\begin{subfigure}{.5\textwidth}
  \centering
  \includegraphics[width=1.0\linewidth]{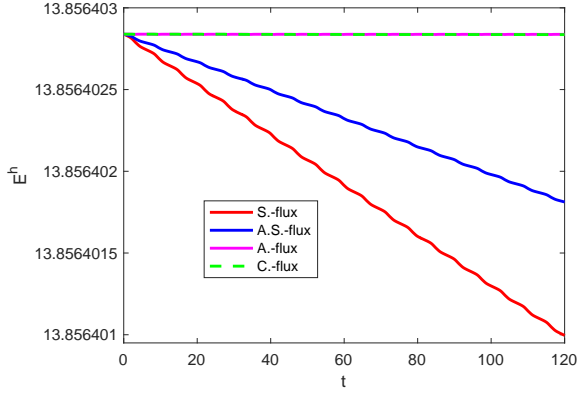}
  \caption{$s = q-1$}
  \label{fig:a1}
\end{subfigure}
\begin{subfigure}{.5\textwidth}
  \centering
  \includegraphics[width=1.0\linewidth]{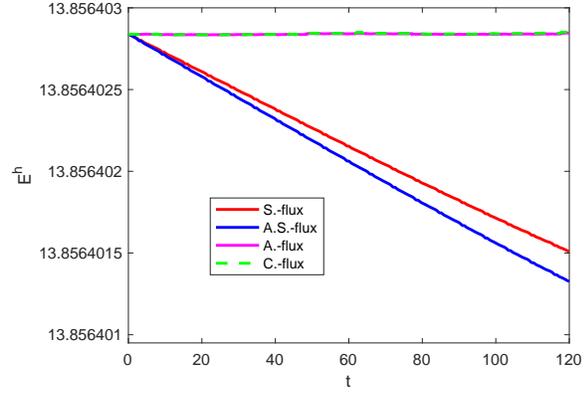}
  \caption{$s = q$}
  \label{fig:b1}
\end{subfigure}
\caption{\scriptsize{Plots of the history of the numerical energy for the standing breather solution.}}
\label{fig:energy_1d_breather}
\end{figure}

In Figure \ref{fig:energy_1d_breather}, we present the numerical energy for the schemes with S.-flux, A.S-flux, A.-flux and C.-flux.  From the left to the right are the cases where $u^h, v^h$ are in different approximation spaces, $s = q - 1$, and the same approximation space, $s = q$, respectively. Overall, we observe that the change of the numerical energy is not significant compared with the initial energy even for dissipating schemes. The A.S-flux and S.-flux produce energy dissipating schemes and they have somewhat different performance depending on $s$, but even then the energy is conserved
to around $7$ digits.

\textbf{\emph{The numerical standing breather soliton:}} the numerical standing breather solutions are shown in Figure \ref{fig:breathers} and Figure \ref{fig:breather}. In the simulation, $u^h, v^h$ are chosen to be in the same approximation space with $q = s =4$ and the S.-flux is used. Figure \ref{fig:breathers} shows both exact and numerical breather solutions at several times, $t = 0, 45, 90, 120$ respectively. Figure \ref{fig:breather} presents the space-time plot of the breather solution from $t = 0$ to $t = 120$. We find that the numerical results match well with the analytic solution.

\begin{figure}[h!]
\begin{subfigure}{.5\textwidth}
  \centering
  \includegraphics[width=1.0\linewidth]{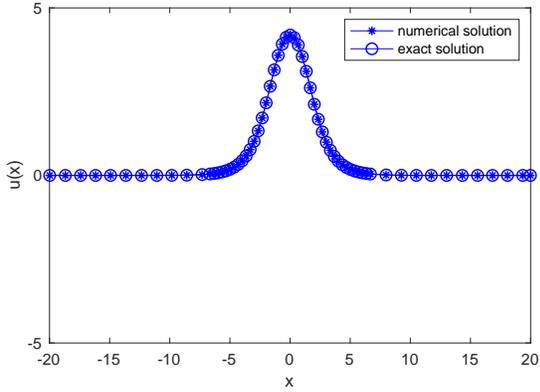}
  \caption{$t = 0$}
  \label{fig:ab1}
\end{subfigure}
\begin{subfigure}{.5\textwidth}
  \centering
  \includegraphics[width=1.0\linewidth]{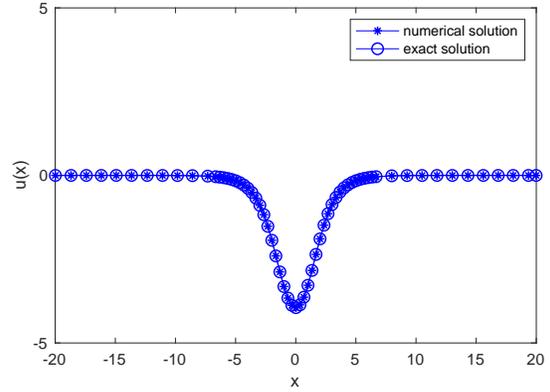}
  \caption{$t = 45$}
  \label{fig:ab2}
\end{subfigure}
\begin{subfigure}{.5\textwidth}
  \centering
  \includegraphics[width=1.0\linewidth]{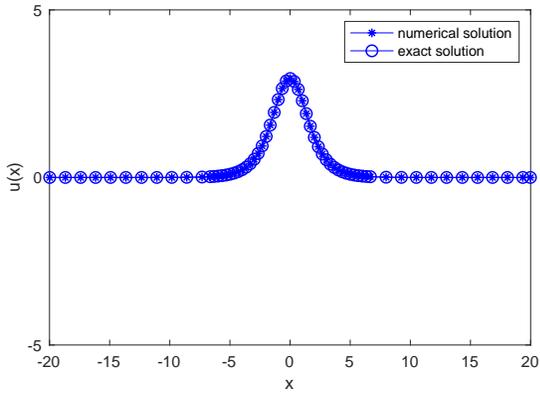}
  \caption{$t = 90$}
  \label{fig:ab3}
\end{subfigure}
\begin{subfigure}{.5\textwidth}
  \centering
  \includegraphics[width=1.0\linewidth]{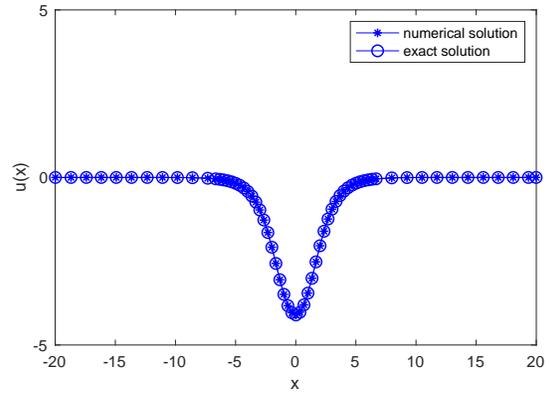}
  \caption{$t = 120$}
  \label{fig:ab4}
\end{subfigure}
\caption{\scriptsize{Plots of the standing breather with the degree of approximation space $q = s=4$. The S.-flux is used in the simulation. From the top to the bottom, the left to the right, the numerical and exact breather solutions at $t = 0,45,90,120$ are plotted.}}
\label{fig:breathers}
\end{figure}

\begin{figure}[h!]
\begin{center}
\includegraphics[width=0.5\textwidth]{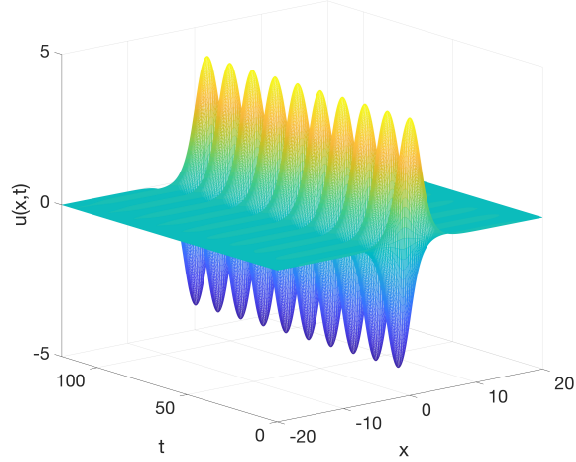}
\caption{\scriptsize{Space-time plots of the standing breather with the degree of approximation space $q = s=4$. The S.-flux is used in the simulation.}}
\label{fig:breather}
\end{center}
\end{figure}

\textbf{\emph{Time history of the $L^2$ error:}} the time history of the $L^2$ errors for the standing breather soliton solution with A.-flux, C.-flux, A.S.-flux and S.-flux are plotted in Figure \ref{fig:breather_errors} for both $s = q$ and $s = q-1$. Particularly, $q$ is set to be $4$ in the numerical simulation. The top panel is for energy-conserving schemes with the A.-flux and C.-flux from the left to right. The bottom panel is for energy-dissipating schemes with the A.S.-flux and S.-flux from the left to right. The error dynamics for all schemes except for the C.-flux are quite similar to each other
and for the two values of $s$ tested. For the C-flux., however, the errors display noticeably different patterns. Nonetheless,
the peak errors for all eight experiments are comparable.
Finally, considering that the standing breather solution is periodic in time, we note that the $L^2$ error grows linearly in time.

\begin{figure}[h!]
\begin{subfigure}{.5\textwidth}
  \centering
  \includegraphics[width=1.0\linewidth]{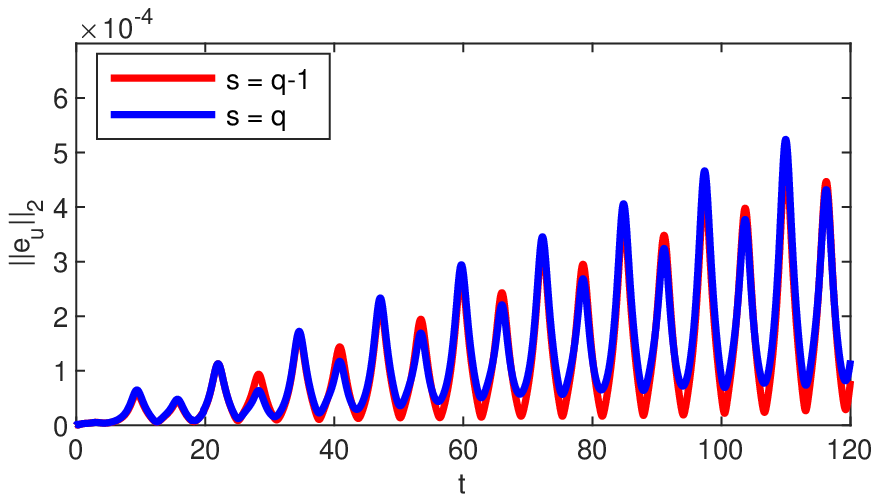}
  \caption{$A.$-flux}
  \label{fig:ae1}
\end{subfigure}
\begin{subfigure}{.5\textwidth}
  \centering
  \includegraphics[width=1.0\linewidth]{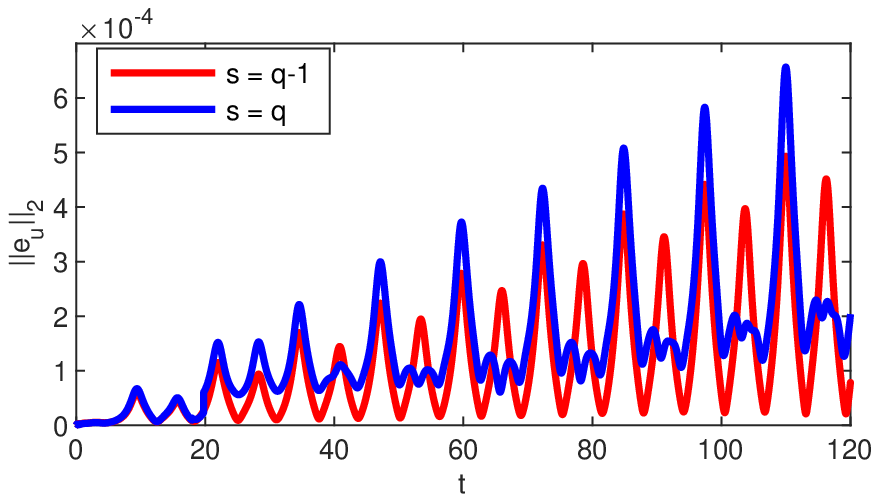}
  \caption{$C.$-flux}
  \label{fig:ae2}
\end{subfigure}
\begin{subfigure}{.5\textwidth}
  \centering
  \includegraphics[width=1.0\linewidth]{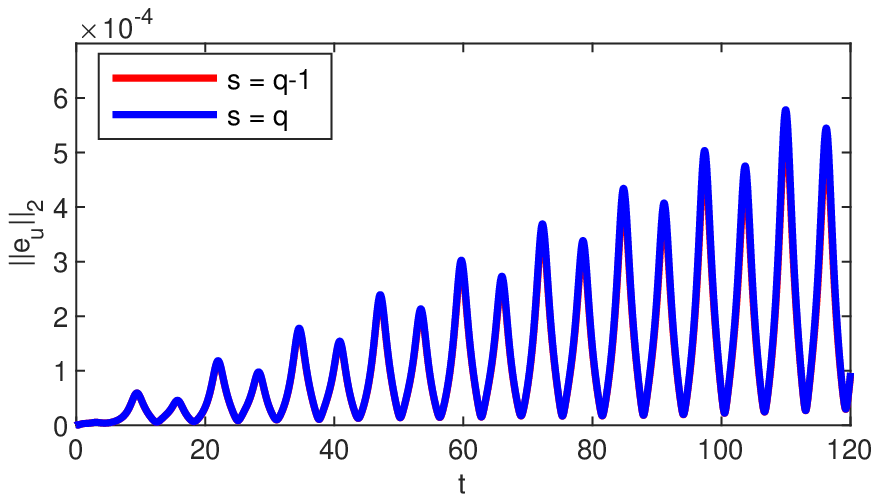}
  \caption{$A.S.$-flux}
  \label{fig:ae3}
\end{subfigure}
\begin{subfigure}{.5\textwidth}
  \centering
  \includegraphics[width=1.0\linewidth]{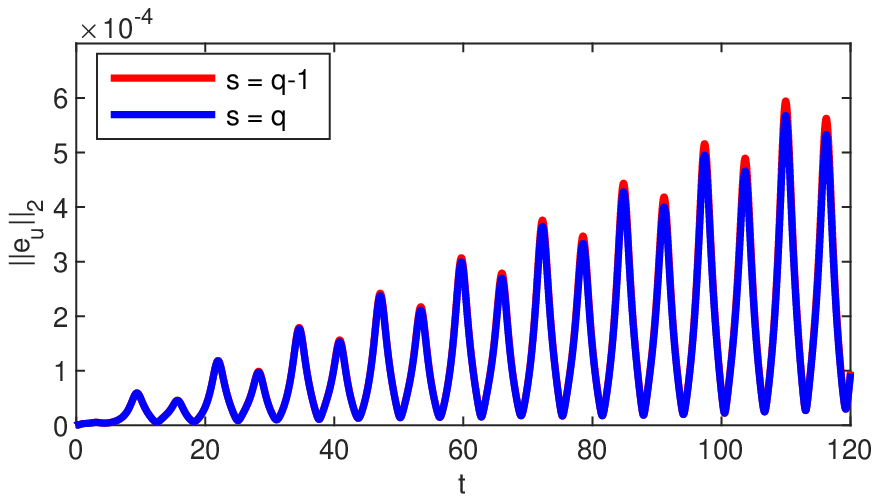}
  \caption{$S.$-flux}
  \label{fig:ae4}
\end{subfigure}
\caption{\scriptsize{Plots of the history of the $L^2$ errors for $u$, standing breather. The first row is for energy-conserving schemes, from the left to right the A.-flux and the C.-flux respectively. The second row is for energy-dissipating schemes, from the left to right the A.S.-flux and the S.-flux respectively. The degree of the approximation space for $u$ is $q = 4$ and for $v$ is $s$.}}
\label{fig:breather_errors}
\end{figure}

\subsubsection{Kink soliton and antikink soliton}\label{sec_kink}
For the kink soliton solution, the sine-Gordon equation (\ref{sine_gordon_1d}) is solved with the initial condition,
\begin{equation*}
 u(x,0) = 4\arctan\left(\mbox{exp}\left(\frac{x}{\sqrt{1-\mu^2}}\right)\right),  \ \ \ \ \frac{\partial u}{\partial t}(x,0) =-\frac{2\mu}{\sqrt{1-\mu^2}}\mbox{sech}\left(\frac{x}{\sqrt{1-\mu^2}}\right), \ \ \ \ x\in(-20,20).
\end{equation*}
The analytic kink solution
\begin{equation}\label{kink}
 u(x,t) = 4\arctan\left(\mbox{exp}\left(\frac{x-\mu t}{\sqrt{1-\mu^2}}\right)\right)
\end{equation}
is a traveling wave increasing monotonically from $0$ to $2\pi$ as $x$ varies from $-\infty$ to $\infty$. In contrast with the kink soliton (\ref{kink}), for the antikink soliton solution we solve the sine-Gordon equation (\ref{sine_gordon_1d}) with the initial conditions,
\begin{equation*}
 u(x,0) = 4\arctan\left(\mbox{exp}\left(-\frac{x}{\sqrt{1-\mu^2}}\right)\right),  \ \ \ \ \frac{\partial u}{\partial t}(x,0) =\frac{2\mu}{\sqrt{1-\mu^2}}\mbox{sech}\left(\frac{x}{\sqrt{1-\mu^2}}\right), \ \ \ \ x\in(-20,20)
\end{equation*}
which leads to an analytic antikink solution
\begin{equation}\label{antikink}
 u(x,t) = 4\arctan\left(\mbox{exp}\left(-\frac{x-\mu t}{\sqrt{1-\mu^2}}\right)\right).
\end{equation}
Compared with the kink solution (\ref{kink}), the antikink soliton (\ref{antikink}) is also a traveling wave solution, but the solution varies monotonically from $2\pi$ to $0$ as $x$ varies from $-\infty$ to $\infty$.

In the numerical simulation, the velocity for both the kink soliton and the antikink soliton is chosen to be $\mu = 0.2$. For the kink soliton, an energy-conserving scheme with the A.-flux is used, while the C.-flux is used in the simulation of
the antikink soliton. We take $u^h$ and $v^h$ to be in different approximation spaces, i.e., $s = q-1$, with $q = 4$. Finally, the problem is evolved with a $4$-stage Runge Kutta time integrator until $T = 80$ with time step size $\Delta t = 0.01$.

The space-time plots of the kink and antikink solitons are shown in Figure \ref{fig:kink_antikink}. From the left to the right are the kink soliton and the antikink soliton respectively. From the left graph, we see that the kink soliton increases monotonically from $0$ to $2\pi$ and the antikink soliton decreases from $2\pi$ to $0$ monotonically in the right graph. Both kink and antikink solitons move from the left to the right and keep their original shape.

\begin{figure}[tbhp]
\begin{center}
\begin{subfigure}{.49\textwidth}
  \centering
  \includegraphics[width=1.0\linewidth]{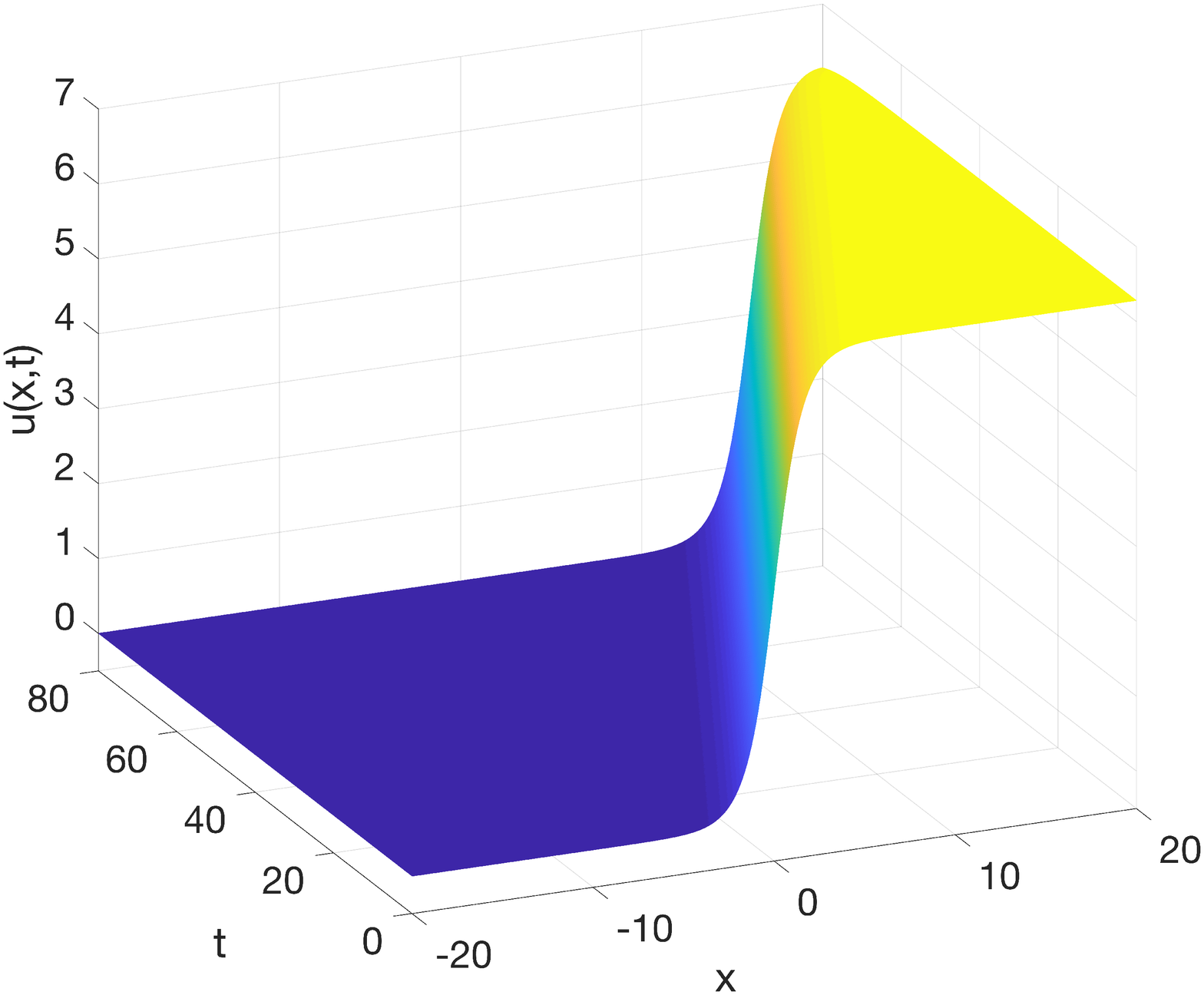}
  \caption{Kink soliton}
  \label{fig:a3}
\end{subfigure}
\begin{subfigure}{.49\textwidth}
  \centering
  \includegraphics[width=1.0\linewidth]{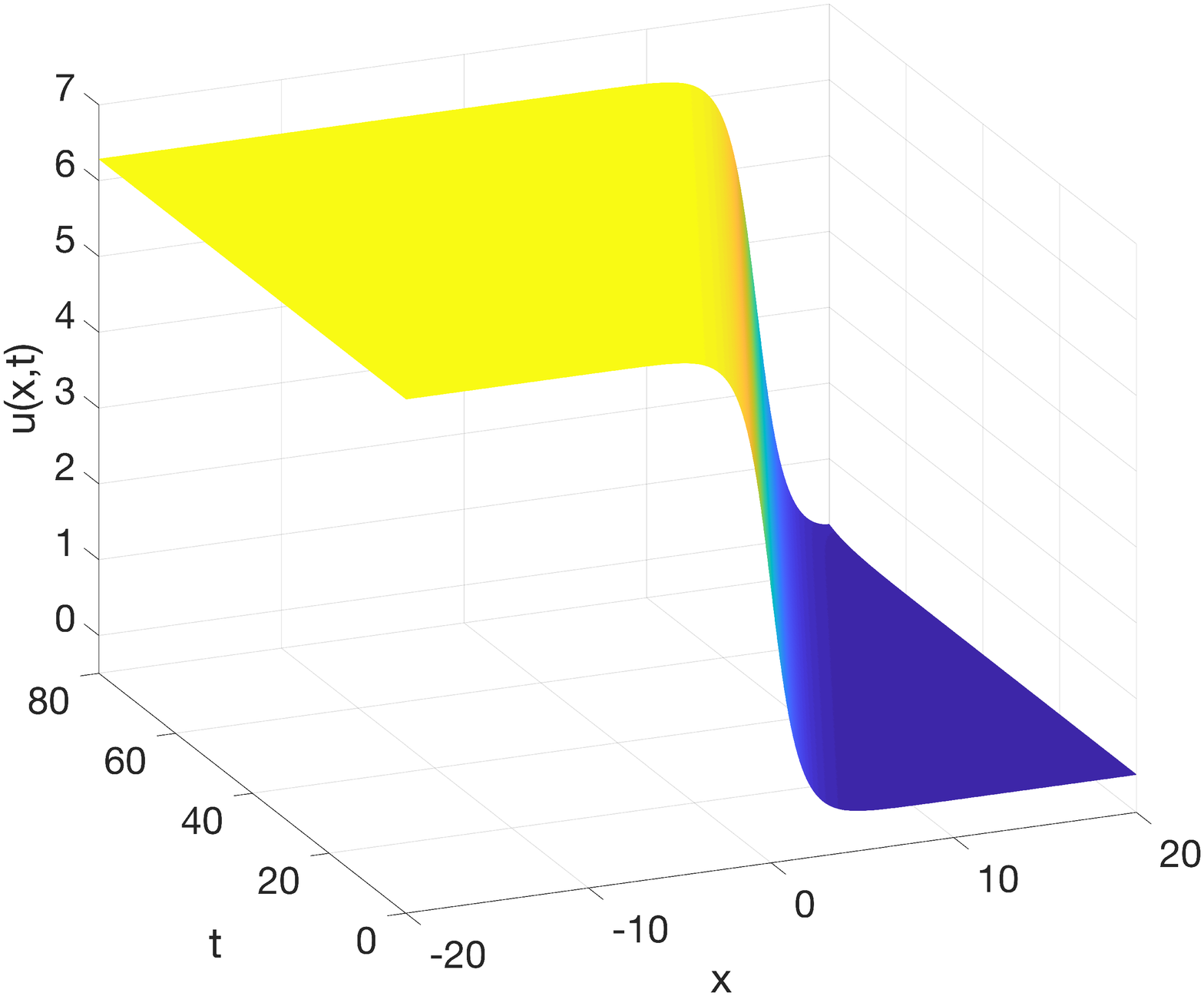}
  \caption{Antikink soliton}
  \label{fig:b3}
\end{subfigure}
\caption{\scriptsize{From the left to the right, plots for the kink and antikink solitons respectively. The degree of approximation space for $u$ is $q = 4$ and for $v$ is $s = 3$. An energy-based DG scheme with the A.-flux is used in the simulation of kink soliton and the C.-flux is used for the simulation of antikink soliton.}}
\label{fig:kink_antikink}
\end{center}
\end{figure}

\subsubsection{Kink-kink collision and kink-antikink collision}
To numerically simulate the kink-kink collision,
we use the superposition of two kink solitons as the initial condition for (\ref{sine_gordon_1d}), one moves from the left to the right and the other moves from the right to the left as follows,
\begin{equation*}
    u(x,0) = 4\arctan\left(\mbox{exp}\left(\frac{x+10}{\sqrt{1-\mu^2}}\right)\right) + 4\arctan\left(\mbox{exp}\left(\frac{x-10}{\sqrt{1-\mu^2}}\right)\right),\ \ \ \ x\in(-20,20),
\end{equation*}
\begin{equation*}
    \frac{\partial u}{\partial t}(x,0) = -\frac{2\mu}{\sqrt{1-\mu^2}}\mbox{sech}\left(\frac{x+10}{\sqrt{1-\mu^2}}\right) + \frac{2\mu}{\sqrt{1-\mu^2}}\mbox{sech}\left(\frac{x-10}{\sqrt{1-\mu^2}}\right),\ \ \ \ x\in(-20,20).
\end{equation*}
Similarly, for the kink-antikink soliton collision we choose the superposition of a kink soliton and an antikink soliton as the initial conditions; the kink soliton moves from the left to the right and the antikink soliton moves from the right to the left as follows,
\begin{equation*}
    u(x,0) = 4\arctan\left(\mbox{exp}\left(\frac{x+10}{\sqrt{1-\mu^2}}\right)\right) + 4\arctan\left(\mbox{exp}\left(-\frac{x-10}{\sqrt{1-\mu^2}}\right)\right),\ \ \ \ x\in(-20,20),
\end{equation*}
\begin{equation*}
    \frac{\partial u}{\partial t}(x,0) = -\frac{2\mu}{\sqrt{1-\mu^2}}\mbox{sech}\left(\frac{x+10}{\sqrt{1-\mu^2}}\right) - \frac{2\mu}{\sqrt{1-\mu^2}}\mbox{sech}\left(-\frac{x-10}{\sqrt{1-\mu^2}}\right),\ \ \ \ x\in(-20,20).
\end{equation*}
Note that we simply use the superposition of two kink solitons (kink and antikink solitons) to be the initial conditions rather than the analytic solution of the corresponding collisions.

The parameter $\mu$ is chosen to be $0.2$ in the numerical simulation. For the kink-kink collision soliton, an energy-dissipating scheme with the A.S.-flux is used, and the S.-flux is used in the simulation of kink-antikink collision soliton. Besides, $u^h, v^h$ are assumed to be in different approximation spaces, i.e., $s = q-1$, with $q = 4$. Finally, the problem is evolved with a $4$-stage Runge-Kutta time integrator until $T = 80$ with time step size $\Delta t = 0.01$.

The plots of the kink-kink and the kink-antikink soliton collisions are shown in Figure \ref{fig:kinkkink_kinkantikink}. In the left graph we observe that initially the two kinks move towards each other at the same speed. The kink with the profile from $0$ to $2\pi$ moves from left to
right and the kink with profile from $2\pi$ to $4\pi$ moves from right to left. After a certain time, they collide with each other and are immediately reflected, keeping their original shape while moving in the opposite direction. The space-time plot of the kink-antikink collision is shown in the right graph. We see that the kink and antikink solitons move towards each other at the same speed. Here the kink with profile from $2\pi$ to $4\pi$ moves left to right and the antikink with profile from $4\pi$ to $2\pi$ moves right to left. After the collision, they move away from each other with their original velocity and direction but changed profiles.

\begin{figure}[h!]
\begin{center}
\begin{subfigure}{.49\textwidth}
  \centering
  \includegraphics[width=1.0\linewidth]{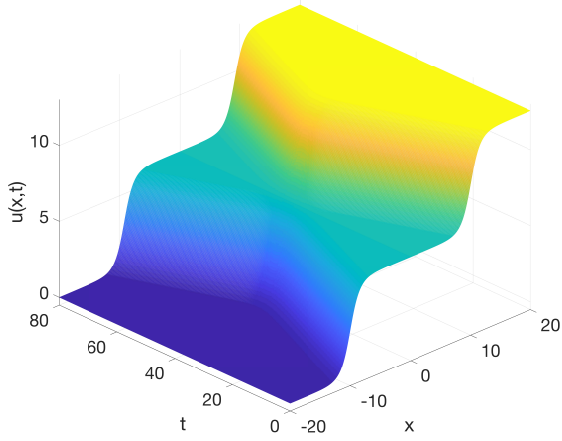}
  \caption{Kink-kink collision}
  \label{fig:a5}
\end{subfigure}
\begin{subfigure}{.49\textwidth}
  \centering
  \includegraphics[width=1.0\linewidth]{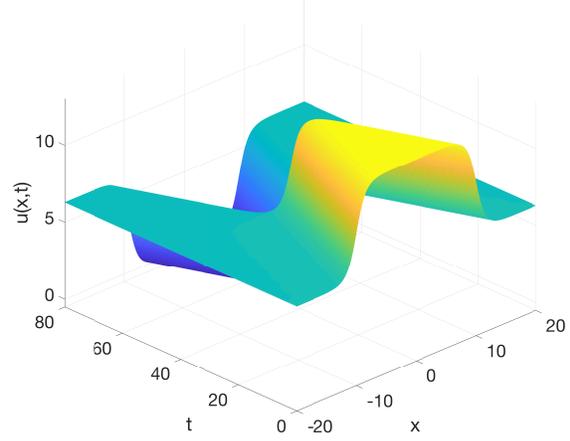}
  \caption{Kink-antikink collision}
  \label{fig:b5}
\end{subfigure}
\caption{\scriptsize{From the left to the right plots of the kink-kink and kink-antikink collisions respectively: the degree of approximation space for $u$ is $q = 4$ and for $v$ is $s = 3$. An energy-based DG scheme with A.S.-flux used in the simulation of kink-kink collision and the S.-flux used for the simulation of the kink-antikink collision.}}
\label{fig:kinkkink_kinkantikink}
\end{center}
\end{figure}

\subsection{Convergence in $2$D}\label{2d_sec_conv}
In this section we investigate the convergence rate of the proposed energy-based DG scheme in $2$D. Specifically, we set $\theta = 0$ and $f(u) = -4u^3$, i.e.,
\begin{equation}\label{2d_eq_conv}
    \frac{\partial^2 u}{\partial t^2} = c^2\left(\frac{\partial^2 u}{\partial x^2}+\frac{\partial^2 u}{\partial y^2}\right) - 4u^3 + f_1(x,y,t), \ \ \ \ (x,y)\in(0,1)\times(0,1),\ \ \ \ t\geq0.
\end{equation}
We construct a manufactured solution
\begin{equation}\label{2d_sol_conv}
    u(x,y,t) = \cos(2\pi x)\cos(2\pi y)\sin(2\pi t),\ \ \ \ (x,y)\in(0,1)\times(0,1),\ \ \ \ t\geq0 ,
\end{equation}
to solve (\ref{2d_eq_conv}). The initial conditions, boundary conditions and external forcing $f_1(x,y,t)$ are determined by $u$ in (\ref{2d_sol_conv}).

The discretization is performed with elements whose vertices are on the Cartesian grids defined by $x_i = ih, y_j = jh, i,j = 0,1,\cdots,n$ with $h = 1/n$. We evolve the solution with the RK4 time integrator until the final time $T = 0.2$ with a time step size of $\Delta t = 0.075h/(2\pi)$. As in the $1$D test, we use four different fluxes: C.-flux, A.-flux, A.S.-flux and S.-flux, but only consider the case where $u^h$ and $v^h$ are in the same approximation space, i.e., $q_x = s_x = q$ and $q_y = s_y = q$.

\begin{figure}[h!]
\begin{subfigure}{.5\textwidth}
  \centering
  \includegraphics[width=1.0\linewidth]{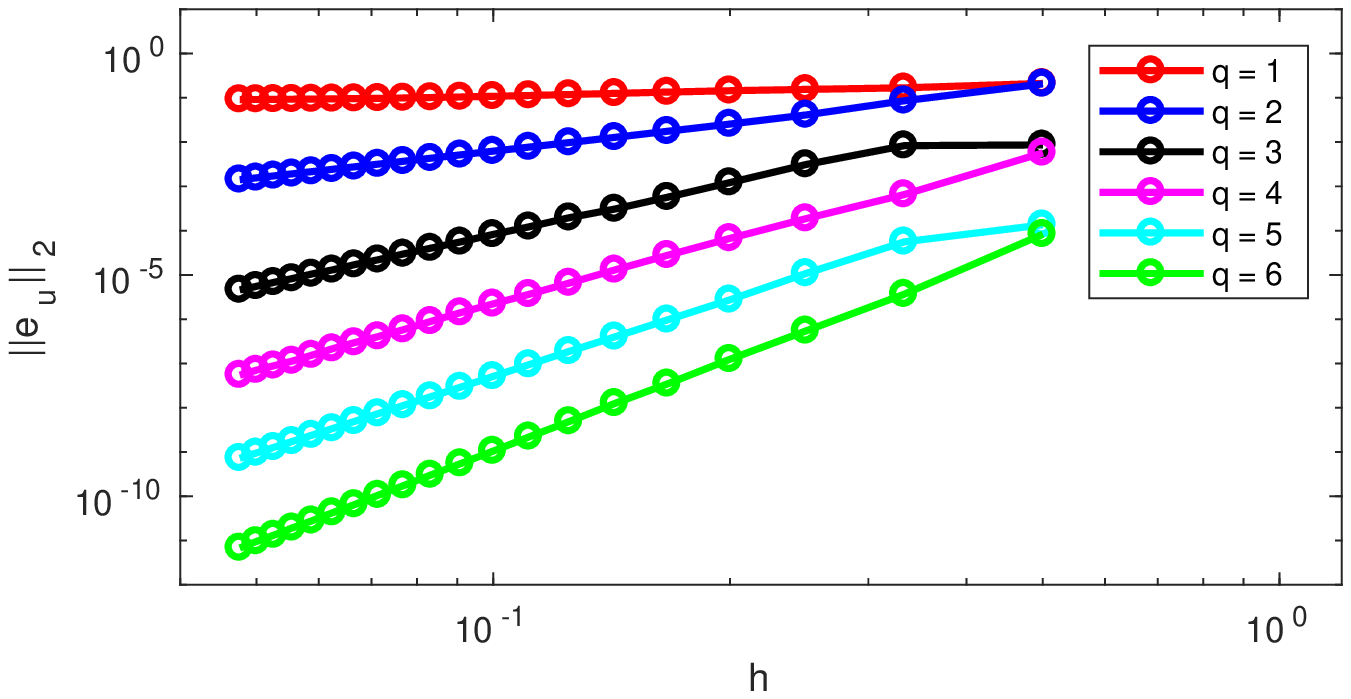}
  \caption{A.-flux}
  \label{fig:a7}
\end{subfigure}
\begin{subfigure}{.5\textwidth}
  \centering
  \includegraphics[width=1.0\linewidth]{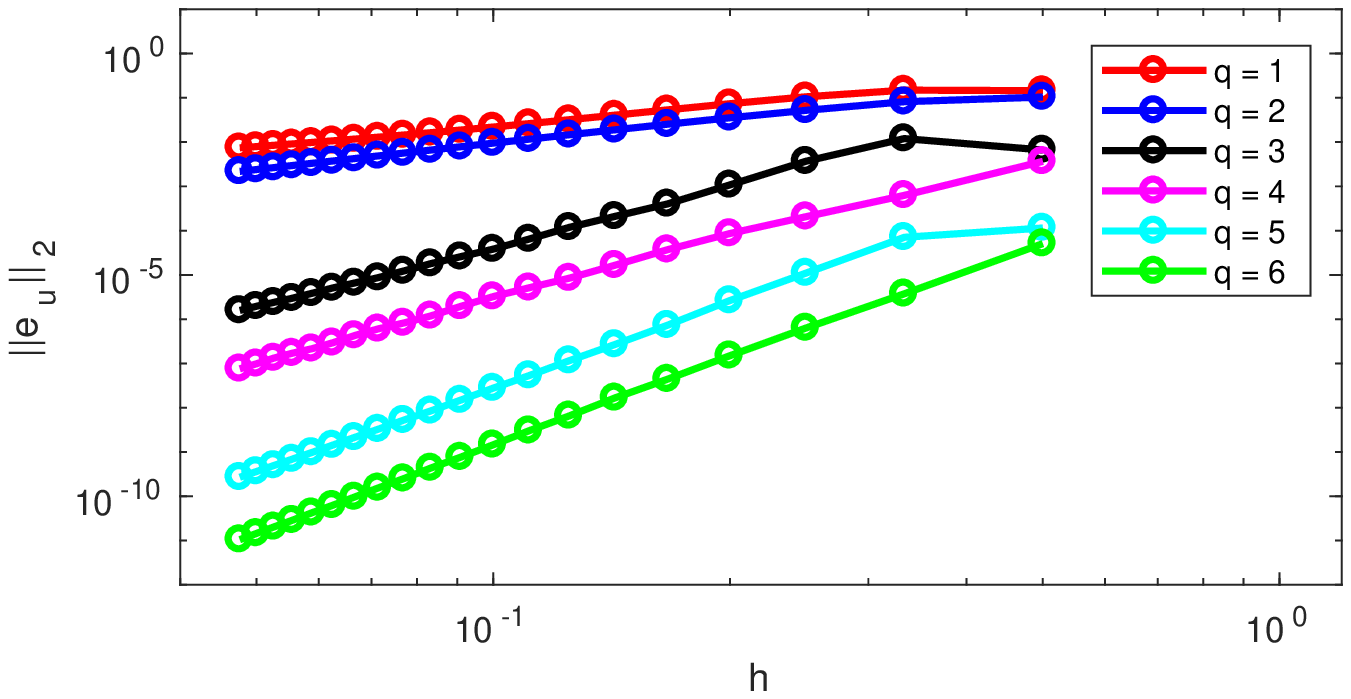}
  \caption{C.-flux}
  \label{fig:b7}
\end{subfigure}
\begin{subfigure}{.5\textwidth}
  \centering
  \includegraphics[width=1.0\linewidth]{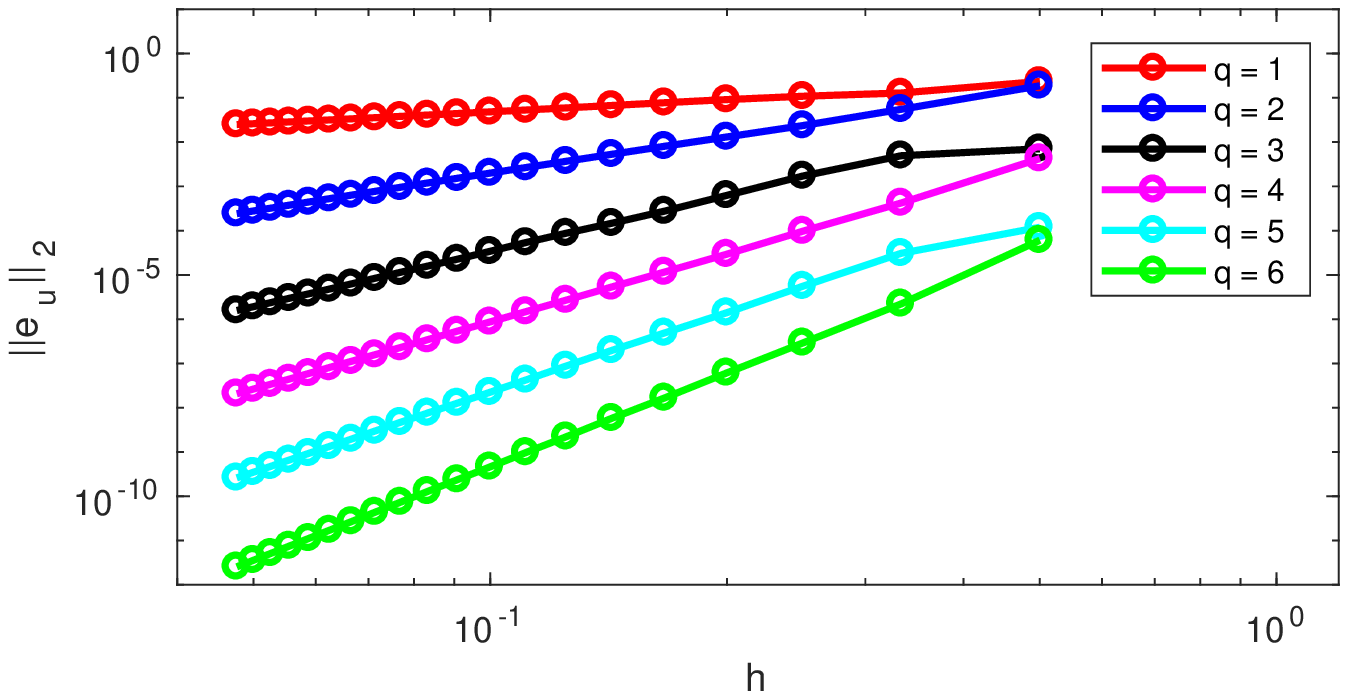}
  \caption{A.S.-flux}
  \label{fig:c7}
\end{subfigure}
\begin{subfigure}{.5\textwidth}
  \centering
  \includegraphics[width=1.0\linewidth]{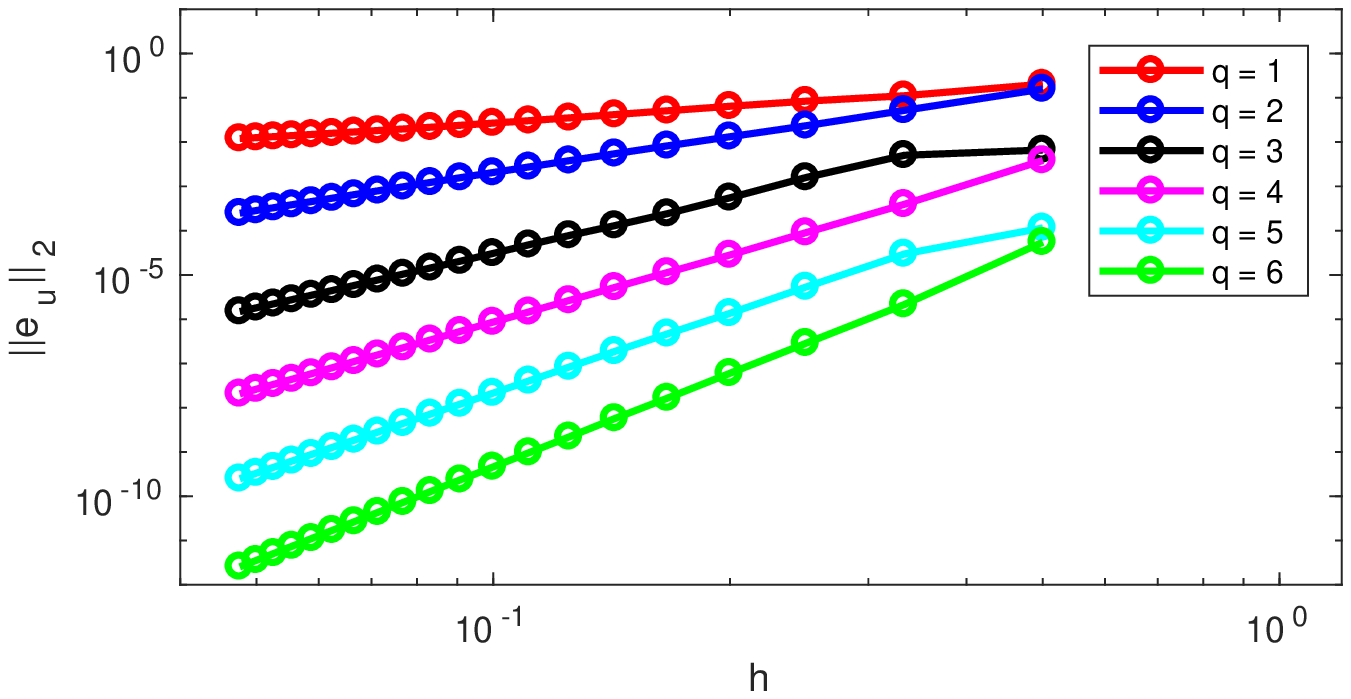}
  \caption{S.-flux}
  \label{fig:d7}
\end{subfigure}
\caption{\scriptsize{The plot of $L^2$ errors for $u$ for the $2$D convergence test: from left to right, top to  bottom A.-flux, C.-flux, A.S.-flux and S.-flux respectively. The approximation degrees for $u^h, v^h$ are $q_x = q_y = s_x = s_y = q$.}}
\label{fig:u_error_2d}
\end{figure}

\begin{table}
\footnotesize
\caption{\scriptsize{Linear regression estimates of the convergence rate for $u$ with C.-flux, A.-flux, S.-flux and A.S.-flux for the $2$D test problem. The approximation degrees for $u^h, v^h$ are $q_x = q_y = s_x = s_y = q$.}}
\begin{center}
  \begin{tabular}{|c|c c c c c c |}
  \hline
{Degree ($q$) of approx. of $u$ }  & 1 & 2 & 3 & 4 & 5 & 6  \\
\hline
{Rate fit with A.-flux ~~~~}&0.19& 2.00 & 3.89 & 4.98 & 5.73 & 6.77 \\
\cline{1-7}
  {Rate fit with C.-flux ~~~~}&1.40 & 1.99 & 4.31 & 4.93 & 6.18 & 6.64 \\
  \cline{1-7}
  {Rate fit with A.S.-flux }&0.89 & 2.86 & 4.14 & 5.03 & 6.03 & 7.02 \\
  \cline{1-7}
  {Rate fit with S.-flux ~~~~}&1.05 & 2.87 & 4.07 & 5.02 & 6.02 & 7.01 \\
  \cline{1-7}
\end{tabular}
\end{center}
\label{convergence_rate_u_2d}
\end{table}

In Figure \ref{fig:u_error_2d}, the $L^2$ errors for $u$ are plotted against the mesh size $h_x = h_y = h$. Table \ref{convergence_rate_u_2d} presents the linear regression estimates of the
convergence rate for $u$ based on the data in Figure \ref{fig:u_error_2d}. Note that we only use the ten finest grids to compute the convergence rate here. From Table \ref{convergence_rate_u_2d}, we observe the optimal convergence rate of $q+1$ for the A.S.-flux and the S.-flux when $q\geq2$ and an order reduction by $1$ compared with the optimal convergence rate for $q = 1$. For the A.-flux and C.-flux, we observe optimal convergence for $q\geq3$, and an order reduction by $1$ for $q = 2$. When $q = 1$, the A.-flux has an order reduction by $2$ compared with the optimal rate and for the C.-flux an order reduction by $\frac{1}{2}$ compared with optimal. These observations are consistent with the results in 1D.

\subsection{Time history of the numerical energy in $2$D}\label{2d_energy}
We compare the numerical energy for both cases with $\theta = 0$ and $\theta = 1$ in this section. Precisely, we consider,
\begin{equation}\label{2d_eq_energy}
   \frac{\partial^2 u}{\partial t^2} + \theta \frac{\partial u}{\partial t} = c^2\left(\frac{\partial^2 u}{\partial x^2}+\frac{\partial^2 u}{\partial y^2}\right) - 4u^3 , \ \ \ \ (x,y)\in(0,1)\times(0,1),\ \ \ \ t>0,
\end{equation}
with initial conditions
\begin{equation*}
    u(x,y,0) = -\cos(2\pi x)\cos(2\pi y), \ \ \ \ \ \frac{\partial u}{\partial t}(x,y,0) = \cos(2\pi x)\cos(2\pi y),\ \ \ \ (x,y)\in(0,1)\times(0,1),
\end{equation*}
and flux free physical boundary conditions,
\begin{equation*}
    \frac{\partial u}{\partial x}(0,y,t) = \frac{\partial u}{\partial x}(1,y,t) = 0, \ \ \ \ y\in(0,1), \ \ \ \ t>0; \ \ \ \  \frac{\partial u}{\partial y}(x,0,t) = \frac{\partial u}{\partial y}(x,1,t) = 0, \ \ \ \ x\in(0,1), \ \ \ \ t>0 .
\end{equation*}
 The space discretization is same as in Section \ref{2d_sec_conv} with $n = 5$. The degree of the approximation space is set to be $q_x = q_y = s_x = s_y = 4$. Finally, the problems are evolved with the RK$4$ time integrator until the final time $T = 10$ with time step size chosen to be $\Delta t = 0.075h/(2\pi)$.

\begin{figure}[h!]
\begin{subfigure}{.5\textwidth}
  \centering
  \includegraphics[width=1.0\linewidth]{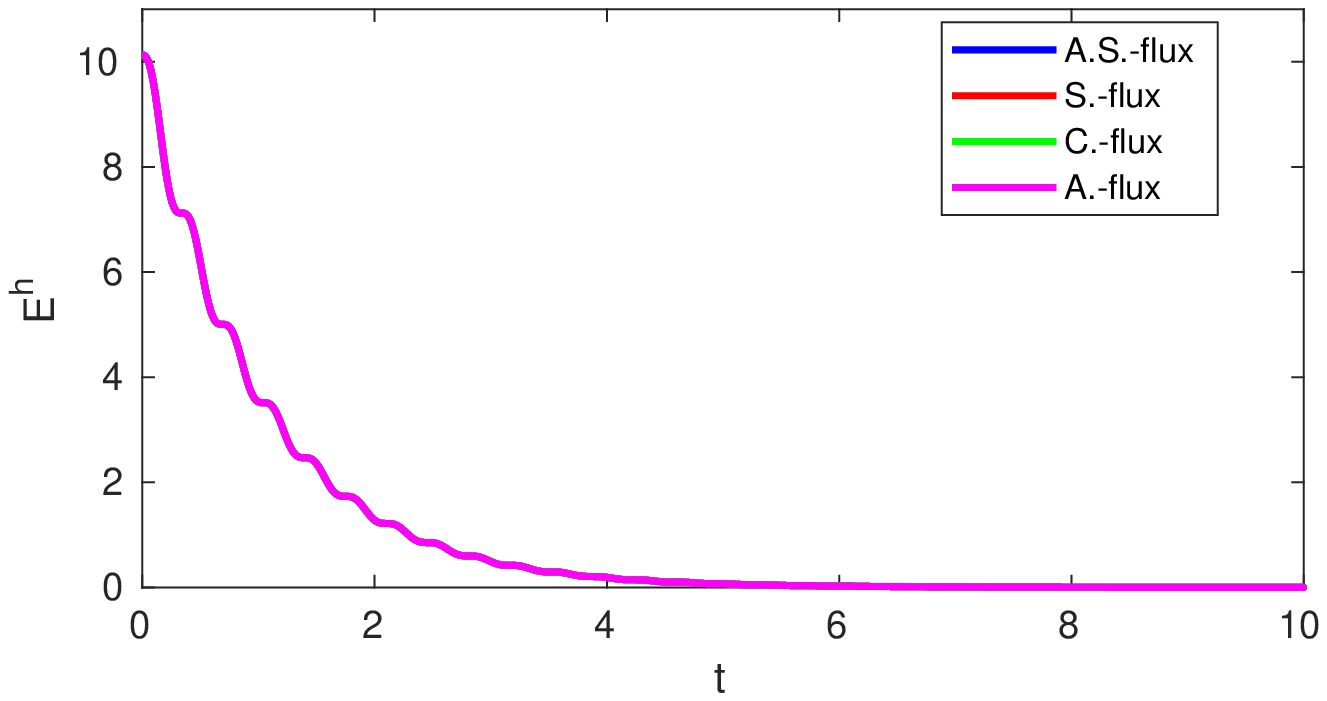}
  \caption{$\theta = 1$}
  \label{fig:a10}
\end{subfigure}
\begin{subfigure}{.5\textwidth}
  \centering
  \includegraphics[width=1.0\linewidth]{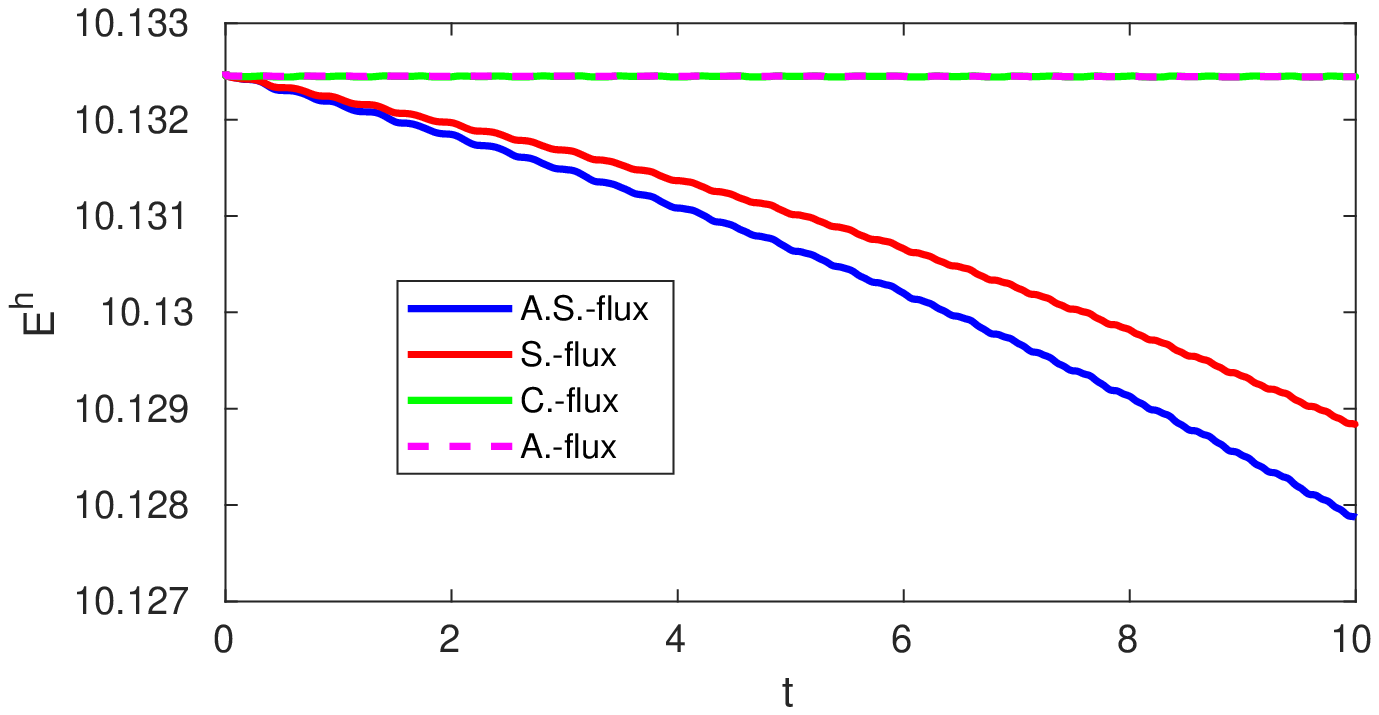}
  \caption{$\theta = 0$}
  \label{fig:b10}
\end{subfigure}
\caption{\scriptsize{The plots of energy for DG solutions of (\ref{2d_eq_energy}) with four different fluxes. For the left graph, the dissipating term is considered, i.e., $\theta = 1$; while the right graph does not contain the dissipating term, i.e., $\theta = 0$.}}
\label{fig:energy_2d}
\end{figure}

In Figure \ref{fig:energy_2d}, the left graph shows the numerical energy evolution with four different fluxes for the problem with dissipating term, $\theta = 1$; the right graph presents the numerical energy evolution with four different fluxes for the problem without dissipating term, $\theta = 0$. We observe that for the case without dissipating term both the A.-flux and C.-flux conserve the numerical energy; both S.-flux and A.S.-flux are energy dissipating but the total dissipation is small. For the case with dissipating term, $\theta = 1$, the numerical energy dissipates for all fluxes, the numerical energy evolution for schemes with A.S.-flux, S.-flux, C.-flux and A.-flux are on top of each other and the numerical energy dissipates very fast.

\subsection{Focusing equation}
Finally, we consider a focusing problem whose energy is indefinite. Specifically, we test the problem
\begin{equation}\label{last_eq}
    \frac{\partial^2 u}{\partial t^2}+\theta\frac{\partial u}{\partial t} = c^2\left(\frac{\partial^2 u}{\partial x^2}+\frac{\partial^2 u}{\partial y^2}\right) + 4u^3, \ \ \ \ (x,y)\in(0,1)\times(0,1),\ \ \ \ t>0,
\end{equation}
for both $\theta = 0$ and $\theta = 1$. We use the same initial data as in Section \ref{2d_energy}
\begin{equation*}
    u(x,y,0) = -\cos(2\pi x)\cos(2\pi y),\ \ \ \ \frac{\partial u}{\partial t}(x,y,0) = \cos(2\pi x)\cos(2\pi y),\ \ \ \ (x,y)\in(0,1)\times(0,1),
\end{equation*}
and periodic boundary conditions are imposed in both $x$ and $y$ directions with $u(0,y,t) = u(1,y,t)$, $y\in(0,1)$ and $u(x,0,t) = u(x,1,t)$, $x\in(0,1)$.

The space discretization is the same with the one in Section \ref{2d_sec_conv} and we set $n = 5$. The degree of the approximation space is set to be $q_x = q_y = s_x = s_y = 4$. Finally, the problems are evolved with the RK$4$ time integrator and the S.-flux with time step size $\Delta t = 0.075h/(2\pi)$.

\begin{figure}[h!]
\begin{subfigure}{.5\textwidth}
  \centering
  \includegraphics[width=1.0\linewidth]{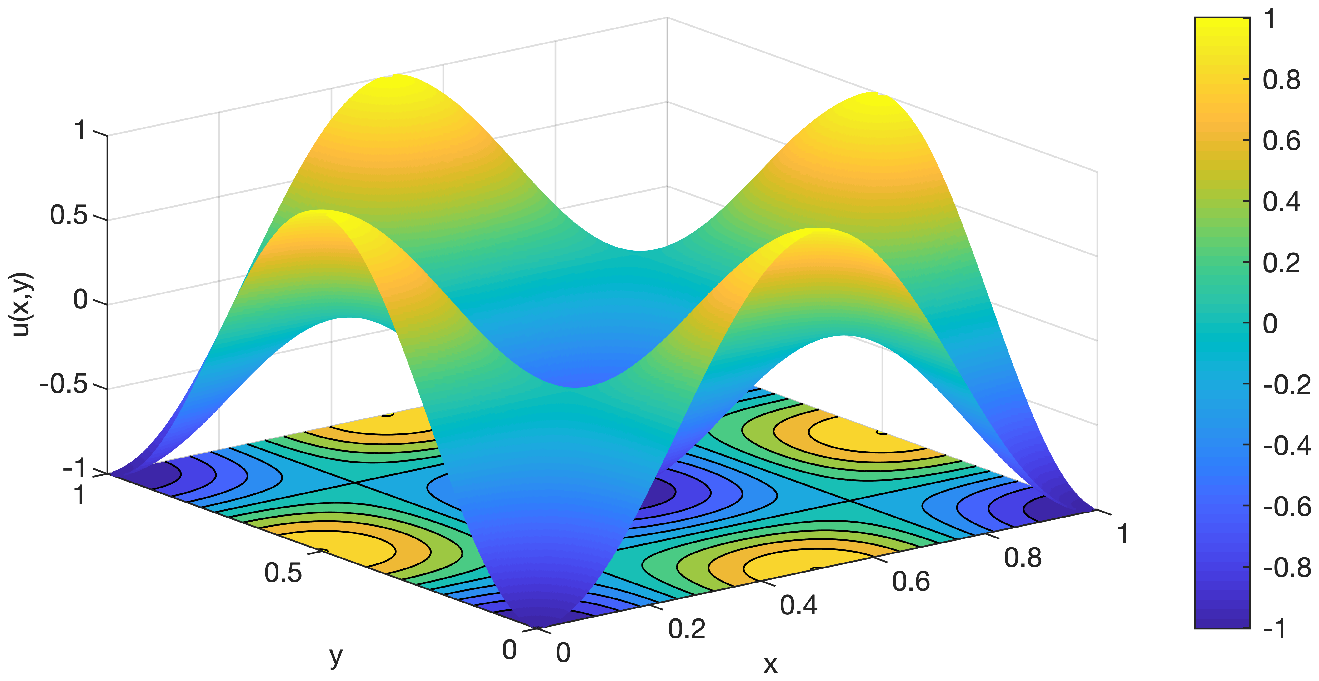}
  \caption{$\theta = 0, t = 0$}
  \label{fig:a11}
\end{subfigure}
\begin{subfigure}{.5\textwidth}
  \centering
  \includegraphics[width=1.0\linewidth]{u_t0.eps}
  \caption{$\theta = 1, t = 0$}
  \label{fig:b11}
\end{subfigure}
\begin{subfigure}{.5\textwidth}
  \centering
  \includegraphics[width=1.0\linewidth]{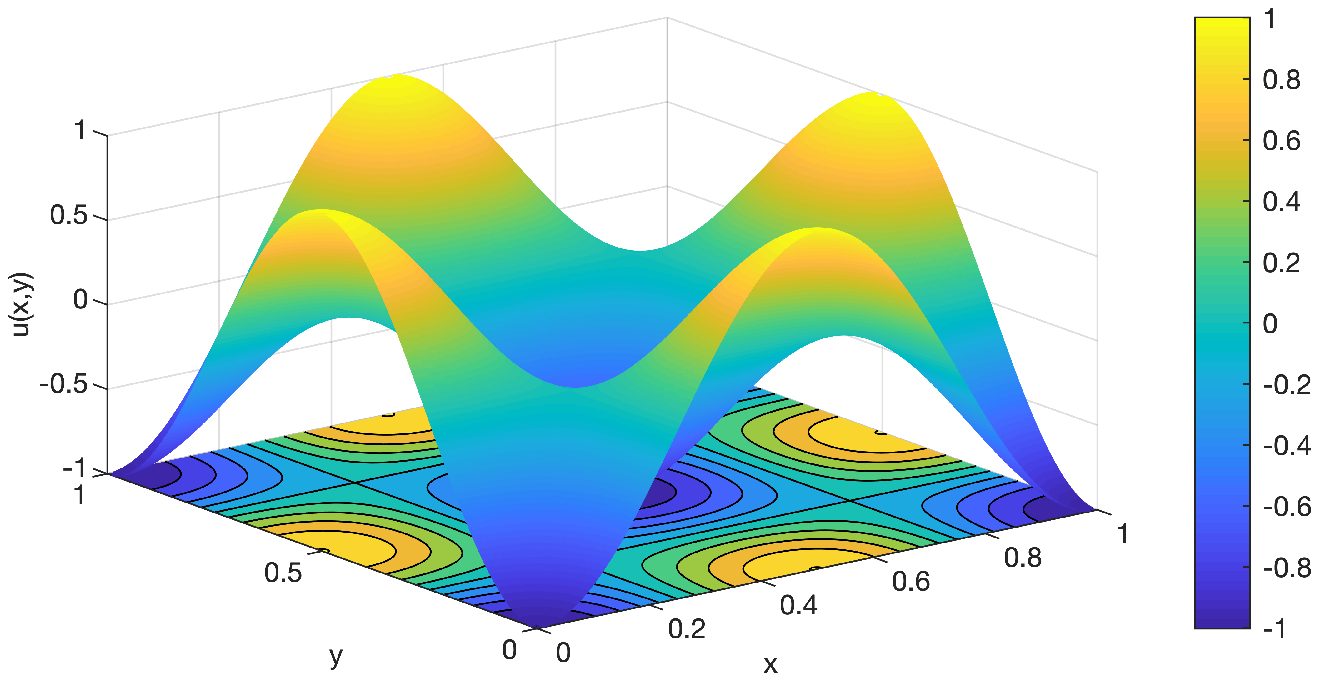}
  \caption{$\theta = 0, t = 0.69$}
  \label{fig:a12}
\end{subfigure}
\begin{subfigure}{.5\textwidth}
  \centering
  \includegraphics[width=1.0\linewidth]{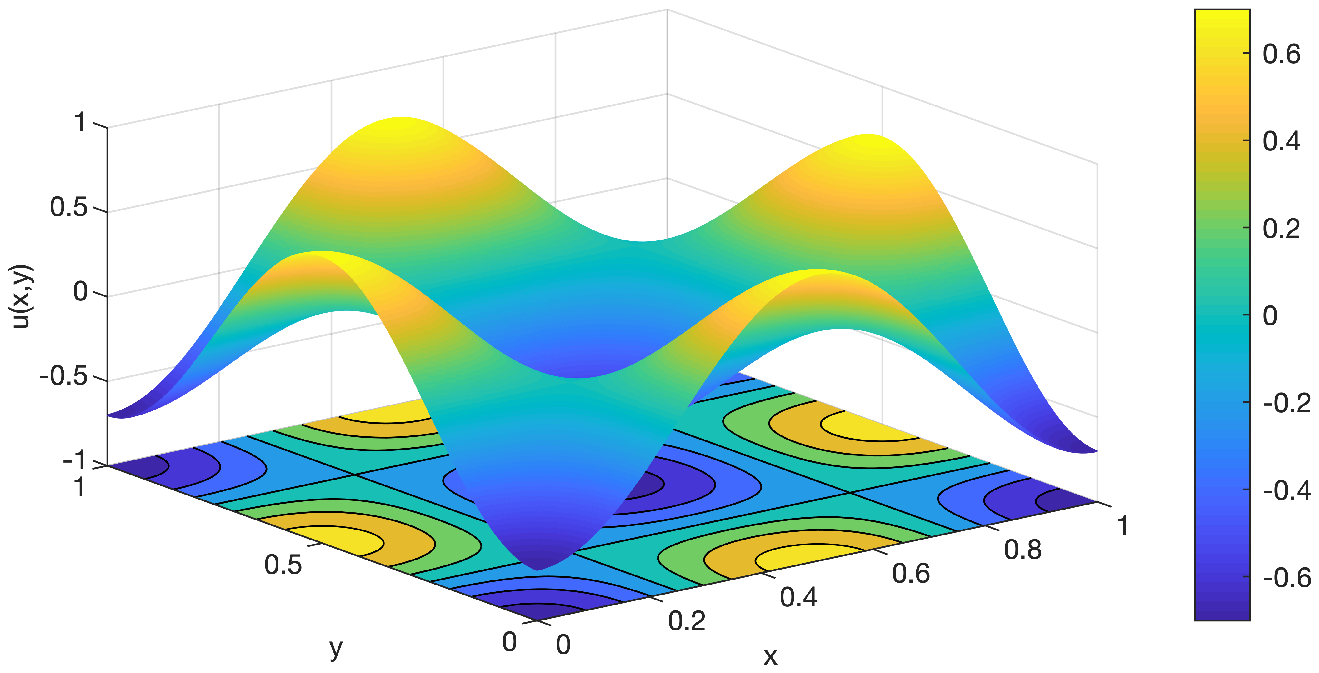}
  \caption{$\theta = 1, t = 0.69$}
  \label{fig:b12}
\end{subfigure}
\begin{subfigure}{.5\textwidth}
  \centering
  \includegraphics[width=1.0\linewidth]{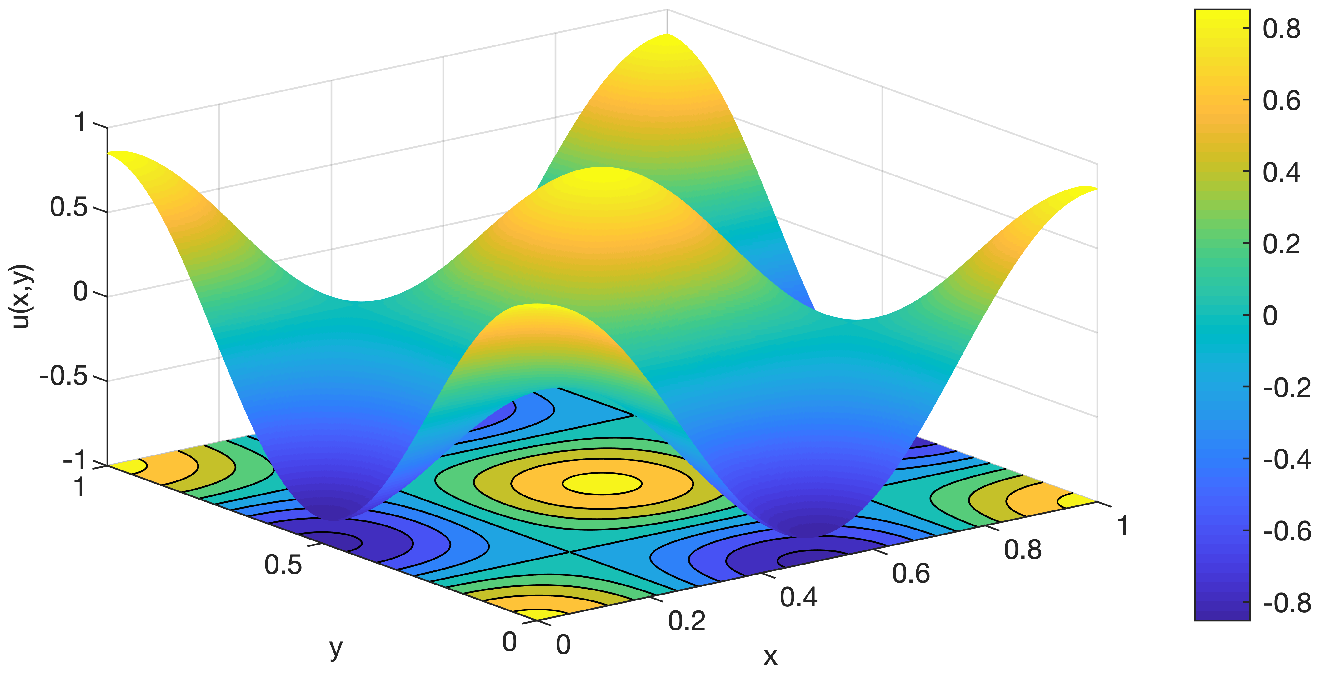}
  \caption{$\theta = 0, t = 6$}
  \label{fig:a13}
\end{subfigure}
\begin{subfigure}{.5\textwidth}
  \centering
  \includegraphics[width=1.0\linewidth]{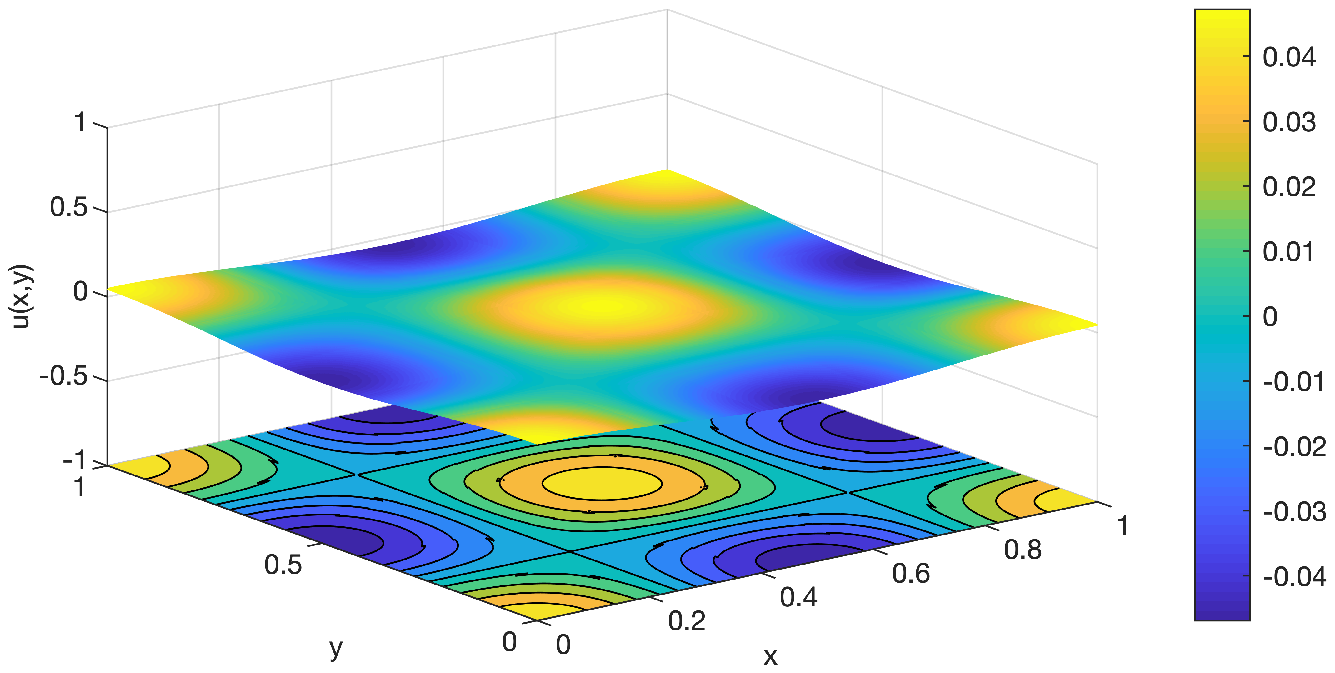}
  \caption{$\theta = 1, t = 6$}
  \label{fig:b13}
\end{subfigure}
\begin{subfigure}{.5\textwidth}
  \centering
  \includegraphics[width=1.0\linewidth]{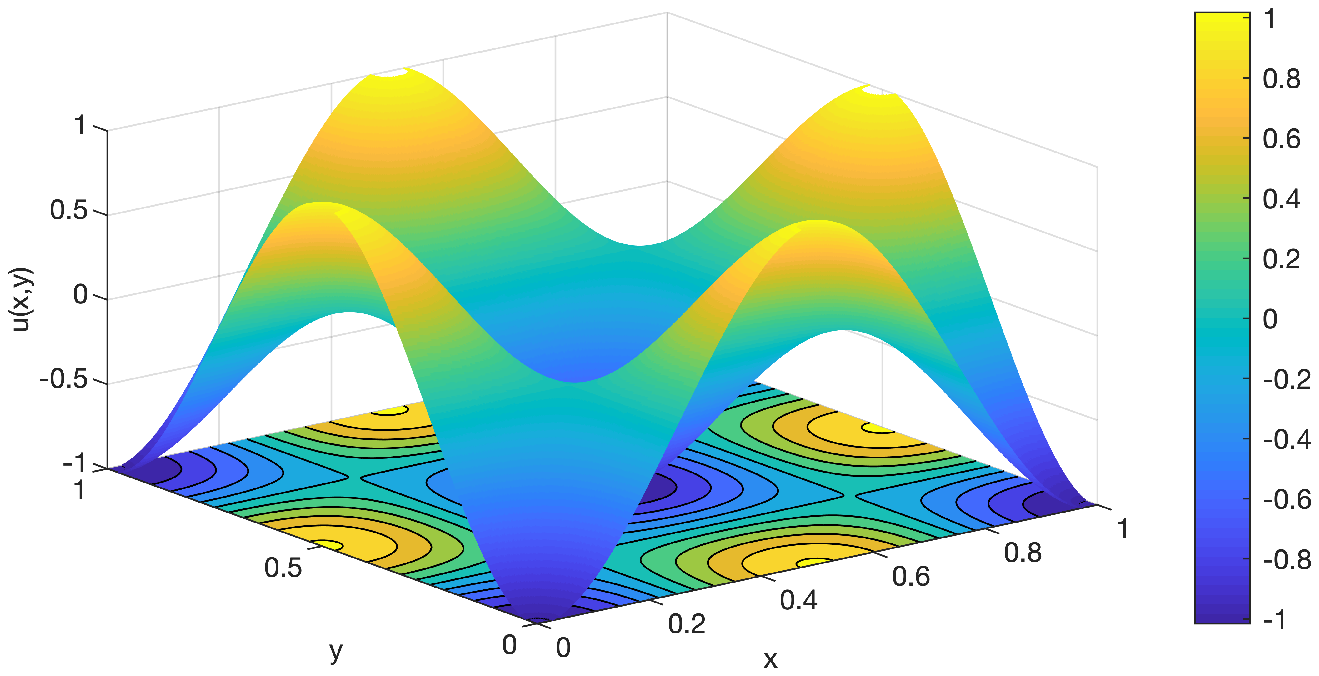}
  \caption{$\theta = 0, t = 10$}
  \label{fig:a14}
\end{subfigure}
\begin{subfigure}{.5\textwidth}
  \centering
  \includegraphics[width=1.0\linewidth]{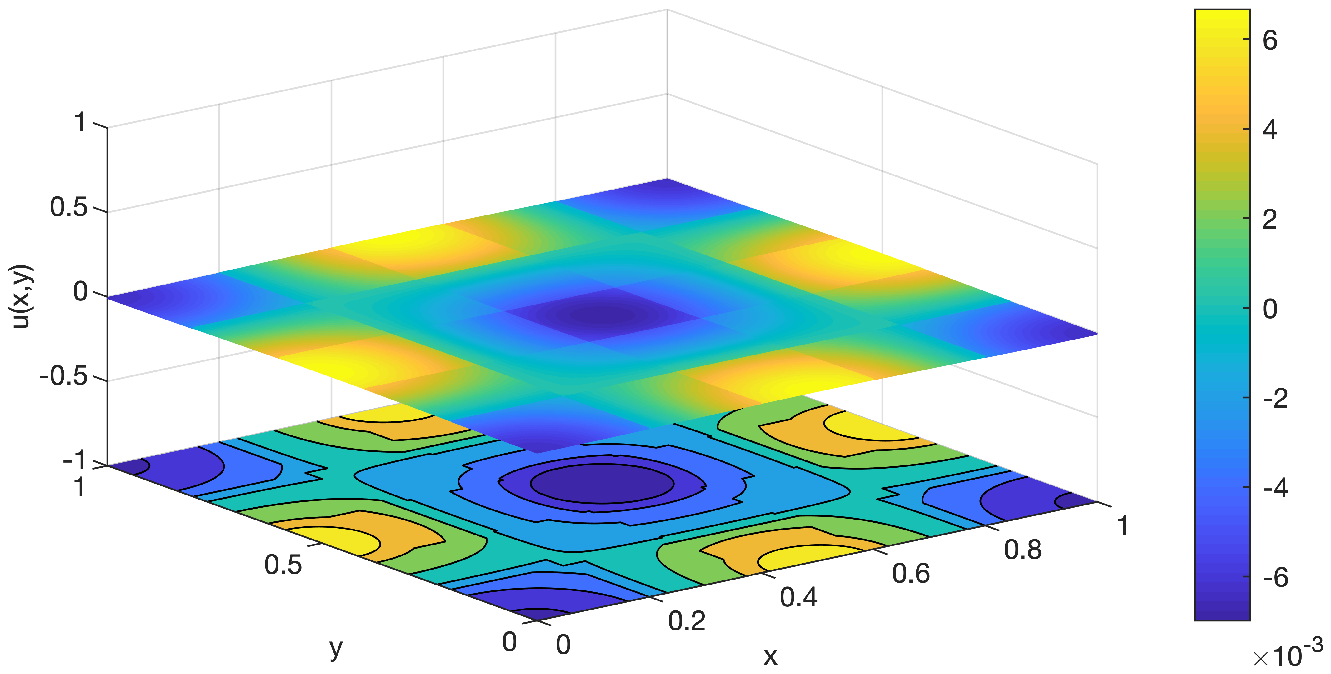}
  \caption{$\theta = 1, t = 10$}
  \label{fig:b14}
\end{subfigure}
\caption{\scriptsize{Plots of $u$ for the focusing equation (\ref{last_eq}) at times $t = 0,0.69,6,10$ with the S.-flux and $q_x = q_y = s_x = s_y = q = 4$. For the left column, $\theta = 0$; for the right column, $\theta = 1$.}}
\label{fig:u_plots_2d}
\end{figure}

Figure \ref{fig:u_plots_2d} shows the time evolution of $u$. From the left column to the right column are for the problem (\ref{last_eq}) without ($\theta = 0$) and with ($\theta = 1$) the dissipating term respectively. On the left column, we observe that the solution $u$ seems to be approximately periodic in time; it recovers its original shape around $t = 0.69$ at first. The right column is for $\theta = 1$, we note that the solution loses its energy as time goes by; at $t = 0.69$, it has a similar shape to the case $\theta = 0$, but the amplitude of the solution is smaller.

\section{Conclusions and future work}

In conclusion, we have demonstrated that the energy-based DG method for
second-order wave equations can be generalized to semilinear problems. In
particular we:
\begin{description}
\item[i.] Modified the weak form proposed in \ci{appelo2015new} so that the
  time derivatives of the approximate solution
  can be computed via the solution of a linear system of
equations in each element,
\item[ii.] Established the stability of the method by proving energy estimates
  for a wide choice of fluxes with mesh-independent
  parametrizations, including energy-conserving central or alternating fluxes
  as well as dissipative upwind fluxes,
\item[iii.] Derived suboptimal estimates of convergence in the energy norm,
\item[iv.] Observed, for polynomial degrees above $3$, optimal convergence in
  the $L^2$ norm for the
  energy-conserving alternating flux as well as for dissipative methods
  based on Sommerfeld flux splitting.
\end{description}

Our main target for future work will be extensions to
systems as well as to problems with more general
nonlinearities. This will enable applications to a wider variety of
problems of physical interest. Here again we plan to exploit the fact that the energy estimates
only depend on the satisfaction of the weak form for certain test functions.

\bibliographystyle{abbrv}
\bibliography{lurefs}

\begin{thebibliography}{10}

\bibitem{appelo2015new}
D.~Appel\"o and T.~Hagstrom.
\newblock A new discontinuous {G}alerkin formulation for wave equations in
  second-order form.
\newblock {\em SIAM J. Num. Anal.}, 53(6):2705--2726, 2015.

\bibitem{el_dg_dath}
D.~Appel\"o and T.~Hagstrom.
\newblock An energy-based discontinuous {G}alerkin discretization of the
  elastic wave equation in second order form.
\newblock {\em Computer Meth. Appl. Mech. Engrg.}, 338:362--391, 2018.

\bibitem{AURSAND2017478}
P.~Aursand and U.~Koley.
\newblock Local discontinuous {G}alerkin schemes for a nonlinear variational
  wave equation modelling liquid crystals.
\newblock {\em J. Comput. Appl. Math.}, 317:478--499, 2017.

\bibitem{GodunovHDG}
T.~Bui-Thanh.
\newblock From {G}odunov to a unified hybridized discontinuous {G}alerkin
  framework for partial differential equations.
\newblock {\em J. Comput. Phys.}, 295:114--146, 2015.

\bibitem{chou2014optimal}
C.-S. Chou, C.-W. Shu and Y.~Xing.
\newblock Optimal energy conserving local discontinuous {G}alerkin methods for
  second-order wave equation in heterogeneous media.
\newblock {\em J. Comput. Phys.}, 272:88--107, 2014.

\bibitem{ciarlet2002finite}
P.~G. Ciarlet.
\newblock {\em The Finite Element Method for Elliptic Problems}, volume~40 of
  {\em Classics in Applied Mathematics}.
\newblock SIAM, 2002.

\bibitem{GSSwave}
M.~Grote, A.~Schneebeli and D.~Sch\"otzau.
\newblock Discontinuous {G}alerkin finite element method for the wave equation.
\newblock {\em SIAM J. Num. Anal.}, 44:2408--2431, 2006.

\bibitem{HesthavenWarburton08}
J.~Hesthaven and T.~Warburton.
\newblock {\em Nodal Discontinuous Galerkin Methods}.
\newblock Number~54 in Texts in Applied Mathematics. Springer-Verlag, New York,
  2008.

\bibitem{RiviereWheelerWave}
B.~Rivi\'ere and M.~Wheeler.
\newblock Discontinuous finite element methods for acoustic and elastic wave
  problems.
\newblock {\em Contemp. Math.}, 329:271--282, 2003.

\bibitem{CiCP-23-747}
N.~Yi and H.~Liu.
\newblock An energy conserving local discontinuous {G}alerkin method for a
  nonlinear variational wave equation.
\newblock {\em Commun. Comput. Phys.}, 23:747--772, 2018.

\bibitem{zhang2019energy}
L.~Zhang, T.~Hagstrom and D.~Appel\"o.
\newblock An energy-based discontinuous {G}alerkin method for the wave equation
  with advection.
\newblock {\em SIAM J. Num. Anal.}, 57:2469--2492, 2019.

\end{thebibliography}

\end{document}